\documentclass[reqno]{amsart}

\input xy
\xyoption{all}
\usepackage{accents}
\usepackage{epsfig}
\usepackage{color}
\usepackage{amsthm}
\usepackage{amssymb}
\usepackage{amsmath}
\usepackage{amscd}
\usepackage{amsopn}
\usepackage{graphicx}

\usepackage{xspace}

\usepackage{url}
\usepackage{enumitem, hyperref}\hypersetup{colorlinks}

\usepackage{color} 

\definecolor{darkred}{rgb}{1,0,0} 
\definecolor{darkgreen}{rgb}{0,0.8,0}
\definecolor{darkblue}{rgb}{0,0,1}

\hypersetup{colorlinks,
linkcolor=darkblue,
filecolor=darkgreen,
urlcolor=darkred,
citecolor=darkgreen}

\makeatletter
\def\reflb#1#2{\begingroup
    #2%
    \def\@currentlabel{#2}%
    \phantomsection\label{#1}\endgroup
}
\makeatother

\numberwithin{equation}{section}
\newtheorem {Theorem}{Theorem}
\numberwithin{Theorem}{section}

\newtheorem {Lemma}[Theorem]    {Lemma}

\newtheorem {Proposition}[Theorem]{Proposition}

\theoremstyle{definition}

\theoremstyle{remark}
\newtheorem{Remark}[Theorem]{Remark}
\newtheorem{Example}[Theorem]{Example}

\def    \eps    {\epsilon}

\newcommand{\CA}{{\mathcal A}}

\newcommand{\CS}{{\mathcal S}}

\newcommand{\supp}{\operatorname{supp}}

\newcommand{\id}{{\mathit id}}

\newcommand{\const}{{\mathit const}}

\newcommand{\fc}{{\mathfrak c}}
\newcommand{\ff}{{\mathfrak f}}

\newcommand{\tPhi}{\tilde{\Phi}}

\newcommand{\A}{{\mathcal A}}

\def    \nat    {{\natural}}

\def    \R      {{\mathbb R}}

\def    \Z      {{\mathbb Z}}
\def    \N      {{\mathbb N}}
\def    \Q      {{\mathbb Q}}

\def    \T      {{\mathbb T}}
\def    \CP     {{\mathbb C}{\mathbb P}}
\def    \RP     {{\mathbb R}{\mathbb P}}

\def    \12    {{\frac{1}{2}}}

\def    \p      {\partial}

\def    \rk     {\operatorname{rk}}

\def    \Sp     {\operatorname{Sp}}
\def    \U     {\operatorname{U}}

\def    \HF     {\operatorname{HF}}
\def    \HC     {\operatorname{HC}}

\def    \Gr     {\operatorname{Gr}}
\def    \H     {\operatorname{H}}
\def    \Tor     {\operatorname{Tor}}

\def    \HM     {\operatorname{HM}}

\def    \bx     {\bar{x}}

\def    \MUCZ  {\operatorname{\mu_{\scriptscriptstyle{CZ}}}}

\begin{document}


\setlength{\smallskipamount}{6pt}
\setlength{\medskipamount}{10pt}
\setlength{\bigskipamount}{16pt}





\title[Conley Conjecture and Beyond]{The Conley Conjecture and
  Beyond}

\author[Viktor Ginzburg]{Viktor L. Ginzburg}
\author[Ba\c sak G\"urel]{Ba\c sak Z. G\"urel}

\address{BG: Department of Mathematics, University of Central Florida,
  Orlando, FL 32816, USA} \email{basak.gurel@ucf.edu}

\address{VG: Department of Mathematics, UC Santa Cruz, Santa Cruz, CA
  95064, USA} \email{ginzburg@ucsc.edu}



\date{\today} 

\thanks{The work is partially supported by NSF grants DMS-1414685 (BG)
  and DMS-1308501 (VG)}

\bigskip

\begin{abstract}
  
  This is (mainly) a survey of recent results on the problem of the
  existence of infinitely many periodic orbits for Hamiltonian
  diffeomorphisms and Reeb flows. We focus on the Conley conjecture,
  proved for a broad class of closed symplectic manifolds, asserting
  that under some natural conditions on the manifold every Hamiltonian
  diffeomorphism has infinitely many (simple) periodic orbits. We
  discuss in detail the established cases of the conjecture and
  related results including an analog of the conjecture for Reeb
  flows, the cases where the conjecture is known to fail, the question
  of the generic existence of infinitely many periodic orbits, and
  local geometrical conditions that force the existence of infinitely
  many periodic orbits. We also show how a recently established
  variant of the Conley conjecture for Reeb flows can be applied to
  prove the existence of infinitely many periodic orbits of a
  low-energy charge in a non-vanishing magnetic field on a surface
  other than a sphere.

\end{abstract}

\maketitle

\tableofcontents

\section{Introduction}
\label{sec:intro}
Hamiltonian systems tend to have infinitely many periodic orbits. For
many phase spaces, every system, without any restrictions, has
infinitely many simple periodic orbits.  Moreover, even if not holding
unconditionally, this is still a $C^\infty$-generic property of
Hamiltonian systems for the majority of phase spaces.  Finally, for
some phase spaces, a system has infinitely many simple periodic orbits
when certain natural local conditions are met.

This paper is mainly a survey focusing on this phenomenon for
Hamiltonian diffeomorphisms and Reeb flows. The central theme of the
paper is the so-called Conley conjecture, proved for a broad class of
closed symplectic manifolds and asserting that under some natural
conditions on the manifold every Hamiltonian diffeomorphism has
infinitely many (simple) periodic orbits. We discuss in detail the
established cases of the conjecture and related results, including an
analog of the conjecture for Reeb flows, and also the manifolds for
which the conjecture is known to fail. In particular, we investigate
local geometrical conditions that force globally the existence of
infinitely many periodic orbits and consider the question of the
generic existence of infinitely many periodic orbits.

We also briefly touch upon the applications to dynamical systems of
physical origin.  For instance, we show how a recently established
variant of the Conley conjecture for Reeb flows can be used to prove
the existence of infinitely many simple periodic orbits of a
low-energy charge in a non-vanishing magnetic field on a surface other
than a sphere.

Our perspective on the problem and the methods we use here are mainly
Morse theoretic, broadly understood, and homological.  In this
framework, the reasons for the existence of periodic orbits lie at the
interplay between local dynamical features of the system and the
global behavior of the homology ``counting'' the periodic orbits,
e.g., Floer, symplectic or contact homology.

This is not the only perspective on the subject. For instance, in
dimensions two and three, one can alternatively use exceptionally
powerful methods of low-dimensional dynamics (see, e.g., \cite{Fr1,
  Fr2, FH, LeC}) and holomorphic curves (see, e.g., \cite{BH,
  HWZ:convex, HWZ03}). In higher dimensions, however, most of the
results on this class of problems rely on homological methods.

It is important to note that for Hamiltonian diffeomorphisms, in
contrast with the classical setting of geodesic flows on a majority of
manifolds as in \cite{GM}, the existence of infinitely many simple
periodic orbits is not forced by the homological growth. Likewise, the
local dynamical features we consider here are usually of different
flavor from, say, homoclinic intersections or elliptic fixed points
often used in dynamics to infer under favorable circumstances the
existence of infinitely many periodic orbits. There is no known single
unifying explanation for the ubiquity of Hamiltonian systems with
infinitely many periodic orbits. Even in the cases where the Conley
conjecture holds, this is usually a result of several disparate
phenomena.

The survey can be read at three levels. First of all, we give a broad
picture, explain the main ideas, results, and conjectures in a
non-technical way, paying attention not only to what has been proved
but also to what is not known. This side of the survey requires very
little background in symplectic and contact topology and dynamics from
the reader. However, we also give the necessary technical details and
conditions when stating the most important results. Although we recall
the relevant definitions in due course, this level of the survey is
intended for a more expert reader. Finally, in several instances, we
attempt to explain the main ideas of the proofs or even to sketch the
arguments. In particular, in Section \ref{sec:pf} we outline the proof
of the Conley conjecture; here we freely use Floer homology and some
other, not entirely standard, symplectic topological tools.

The survey is organized as follows. In Section \ref{sec:CC}, we
discuss the Conley conjecture (its history, background, and the state
of the art) and the generic existence results, and also set the
conventions and notation used throughout the paper. A detailed outline
of the proof is, as has been mentioned above, given in Section
\ref{sec:pf}. The rest of the paper is virtually independent of this
section. We discuss the Conley conjecture and other related phenomena
for Reeb flows in Section \ref{sec:reeb} and applications of the
contact Conley conjecture to twisted geodesic flows, which govern the
motion of a charge in a magnetic field, in Section \ref{sec:magnetic}.
Finally, in Section \ref{sec:beyond}, we turn to the manifolds for
which the Conley conjecture fails and, taking the celebrated Frank's
theorem (see \cite{Fr1, Fr2}) as a starting point, show how certain
local geometrical features of a system can force the existence of
infinitely many periodic orbits. Here we also briefly touch upon the
problem of the existence of infinitely many simple periodic orbits for
symplectomorphisms and for some other types of ``Hamiltonian''systems.


\section{Conley conjecture}
\label{sec:CC}
\subsection{History and background}
\label{sec:CC-history}
As has been pointed out in the introduction, for many closed
symplectic manifolds, every Hamiltonian diffeomorphism has infinitely
many simple periodic orbits and, in fact, simple periodic orbits of
arbitrarily large period whenever the fixed points are isolated. This
unconditional existence of infinitely many periodic orbits is often
referred to as the Conley conjecture. The conjecture was, indeed,
formulated by Charles Conley in 1984 for tori, \cite{Co}, and since
then it has been a subject of active research focusing on establishing
the existence of infinitely many periodic orbits for broader and
broader classes of symplectic manifolds or Hamiltonian
diffeomorphisms.

The Conley conjecture was proved for the so-called weakly
non-degenerate Hamiltonian diffeomorphisms in \cite{SZ} (see also
\cite{CZ:conley}) and for all Hamiltonian diffeomorphisms of surfaces
other than $S^2$ in \cite{FH} (see also \cite{LeC}). In its original
form for the tori, the conjecture was established in \cite{Hi} (see
also \cite {Ma:SDM}), and the case of an arbitrary closed,
symplectically aspherical manifold was settled in \cite{Gi:CC}. The
proof was extended to rational, closed symplectic manifolds $M$ with
$c_1(TM)|_{\pi_2(M)}=0$ in \cite{GG:gaps}, and the rationality
requirement was then eliminated in \cite{He:irr}. In fact, after
\cite{SZ}, the main difficulty in establishing the Conley conjecture
for more and more general manifolds with aspherical first
Chern class, overcome in this series of works, lied in proving the
conjecture for totally degenerate Hamiltonian diffeomorphisms not
covered by \cite{SZ}.  (The internal logic in \cite{FH,LeC}, relying
on methods from low-dimensional dynamics, was somewhat different.)
Finally, in \cite{GG:nm} and \cite{CGG}, the Conley conjecture was
proved for negative monotone symplectic manifolds. (The main new
difficulty here was in the non-degenerate case.)

Two other variants of the Hamiltonian Conley conjecture have also been
investigated. One of them is the existence of infinitely many periodic
orbits for Hamiltonian diffeomorphisms with displaceable support; see,
e.g., \cite{FS, Gu, HZ, Schwarz, Vi:gen}. Here the form $\omega$ is
usually required to be aspherical, but the manifold $M$ is not
necessarily closed.  The second one is the Lagrangian Conley
conjecture or, more generally, the Conley conjecture for Hamiltonians
with controlled behavior at infinity on cotangent bundles (see
\cite{He:cot, Lo:tori, Lu, Ma}) or even some twisted cotangent bundles
(see \cite{FMP}). In this survey, however, we focus mainly on the case
of closed manifolds.

The Conley conjecture looks deceptively similar to the well-known
conjecture that every closed simply connected Riemannian manifold
(e.g., $S^n$) carries infinitely many non-trivial closed
geodesics. However, this appears to be a very different phenomenon
than the Conley conjecture, for the latter does not distinguish
trivial and non-trivial orbits. For instance, the proof of the
Lagrangian Conley conjecture for the pure kinetic energy Hamiltonian
simply detects the constant geodesics.  We will further discuss the
connection between the two conjectures in Sections \ref{sec:reeb}
and~\ref{sec:beyond}.

What makes the Conley conjecture difficult and even counterintuitive
from a homological perspective is that there seems to be no obvious
homological reason for the existence of infinitely many simple
periodic orbits. As we have already mentioned, there is no homological
growth: the Floer homology of a Hamiltonian diffeomorphism does not
change under iterations and remains isomorphic, up to a Novikov ring,
to the homology of the manifold. (In that sense, the difficulty
\emph{is} similar to that in proving the existence of infinitely many
non-trivial closed geodesics on, say, $S^n$ where the rank of the
homology of the free loop space remains bounded as a function of the
degree.)
 
Ultimately, one can expect the Conley conjecture to hold for the
majority of closed symplectic manifolds.  There are, however, notable
exceptions. The simplest one is $S^2$: an irrational rotation of $S^2$
about the $z$-axis has only two periodic orbits, which are also the
fixed points; these are the poles. In fact, any manifold that admits a
Hamiltonian torus action with isolated fixed points also admits a
Hamiltonian diffeomorphism with finitely many periodic orbits. For
instance, such a diffeomorphism is generated by a generic element of
the torus. In particular, flag manifolds (hence the complex projective
spaces and the Grassmannians), and, more generally, most of the
coadjoint orbits of compact Lie groups as well as symplectic toric
manifolds all admit Hamiltonian diffeomorphisms with finitely many
periodic orbits. In dimension two, there are also such examples with
interesting dynamics. Namely, there exist area preserving
diffeomorphisms of $S^2$ with exactly three ergodic
measures: two fixed points and the area form; see \cite{AK} and, e.g.,
\cite{FK}. These are the so-called pseudo-rotations. By taking direct
products of pseudo-rotations, one obtains Hamiltonian diffeomorphisms
of the products of several copies of $S^2$ with finite number of
ergodic measures, and hence with finitely many periodic orbits. It
would be extremely interesting to construct a Hamiltonian analog of
pseudo-rotations for, say, $\CP^2$.

In all known examples of Hamiltonian diffeomorphisms with finitely
many periodic orbits, all periodic orbits are fixed points, i.e., no
new orbits are created by passing to the iterated diffeomorphisms,
cf.\ Section \ref{sec:beyond}. Furthermore, all such Hamiltonian
diffeomorphisms are non-degenerate, and the number of fixed points is
exactly equal to the sum of Betti numbers.  Note also that Hamiltonian
diffeomorphisms with finitely many periodic orbits are extremely
non-generic; see \cite{GG:generic} and Section~\ref{sec:CC-results}.

In any event, the class of manifolds admitting ``counterexamples'' to
the Conley conjecture appears to be very narrow, which leads one to
the question of finding further sufficient conditions for the Conley
conjecture to hold. There are several hypothetical candidates.  One of
them, conjectured by the second author of this paper, is that the
minimal Chern number $N$ of $M$ is sufficiently large, e.g., $N>\dim
M$. (The condition $c_1(TM)|_{\pi_2(M)}=0$ corresponds to $N=\infty$.)
More generally, it might be sufficient to require the Gromov--Witten
invariants of $M$ to vanish, as suggested by Michael Chance and Dusa
McDuff, or even the quantum product to be undeformed. No results in
these directions have been proved to date. Note also that for all
known ``counterexamples'' to the Conley conjecture $\H_*(M;\Z)$ is
concentrated in even degrees.

Another feature of Hamiltonian diffeomorphisms with finitely many
periodic orbits is that, for many classes of manifolds, the actions or
the actions and the mean indices of their simple periodic orbits must
satisfy certain resonance relations of Floer homological nature; see
\cite{CGG, GG:gaps, GK09, Ke:JMD}.  (There are also analogs of such
resonance relations for Reeb flows, which we will briefly touch upon
in Section \ref{sec:reeb}.)  Although the very existence of
homological resonance relations in the Hamiltonian setting is an
interesting, new and unexpected phenomenon, and some of the results
considered here do make use of these relations, their discussion is
outside the scope of this paper.

\subsection{Results: the state of the art} 
\label{sec:CC-results}
In this section, we briefly introduce our basic conventions and
notation and then state the most up-to-date results on the Conley
conjecture and generic existence of infinitely many periodic orbits
for Hamiltonian diffeomorphisms.

\subsubsection{Conventions and notation} 
\label{sec:conventions} 
Let us first recall the relevant terminology, some of which have
already been used in the previous section. A closed symplectic
manifold $(M^{2n},\omega)$ is said to be \emph{monotone (negative
  monotone)} if $[\omega]|_{\pi_2(M)}=\lambda c_1(TM)|_{\pi_2(M)}$ for
some non-negative (respectively, negative) constant $\lambda$ and
\emph{rational} if $\left<[\omega], {\pi_2(M)}\right>=\lambda_0\Z$,
i.e., the integrals of $\omega$ over spheres in $M$ form a discrete
subgroup of $\R$.  The positive generator $N$ of the discrete subgroup
$\left<c_1(TM),\pi_2(M)\right>\subset \R$ is called the \emph{minimal
  Chern number} of $M$.  When this subgroup is zero, we set
$N=\infty$.  A manifold $M$ is called \emph{symplectic CY
  (Calabi--Yau)} if $c_1(M)|_{\pi_2(M)} =0$ and \emph{symplectically
  aspherical} if $c_1(TM)|_{\pi_2(M)} =0=[\omega]|_{\pi_2(M)}$. A
symplectically aspherical manifold is monotone, and a monotone or
negative monotone manifold is rational.

All Hamiltonians $H$ considered in this paper are assumed to be
$k$-periodic in time (i.e., $H$ is a function $S^1_k\times M\to\R$,
where $S^1_k=\R/k\Z$) and the period $k$ is always a positive integer.
When the period $k$ is not specified, it is equal to one, and
$S^1=\R/\Z$. We set $H_t = H(t,\cdot)$ for $t\in S^1_k$. The
(time-dependent) Hamiltonian vector field $X_H$ of $H$ is defined by
$i_{X_H}\omega=-dH$.  A \emph{Hamiltonian diffeomorphism} is the
time-one map, denoted by $\varphi_H$ or just $\varphi$, of the
time-dependent Hamiltonian flow (i.e., Hamiltonian isotopy)
$\varphi_H^t$ generated by $X_H$. It is preferable throughout this
section to view $\varphi$ as an element, determined by $\varphi_H^t$,
of the universal covering of the group of Hamiltonian diffeomorphisms.
A one-periodic Hamiltonian $H$ can also be treated as $k$-periodic. In
this case, we will use the notation $H^{\nat k}$ and, abusing
terminology, call $H^{\nat k}$ the $k$th iteration of $H$.

In what follows, we identify the periodic orbits of $H$ (i.e., of
$\varphi_H^t$) with integer period $k$ and periodic orbits of
$\varphi$.  A periodic orbit $x$ of $H$ is \emph{non-degenerate} if
the linearized return map $d\varphi|_{x} \colon T_{x(0)}M\to
T_{x(0)}M$ has no eigenvalues equal to one. Following \cite{SZ}, we
call $x$ \emph{weakly non-degenerate} if at least one of the
eigenvalues is different from one and \emph{totally degenerate}
otherwise.  Finally, a periodic orbit is said to be \emph{strongly
  non-degenerate} if no roots of unity are among the eigenvalues of
$d\varphi|_{x}$.  This terminology carries over to Hamiltonians $H$
and Hamiltonian diffeomorphisms $\varphi$.  For instance, $\varphi$ is
non-degenerate if all its one-periodic orbits are non-degenerate and
strongly non-degenerate if all iterations $\varphi^k$ are
non-degenerate, etc.

\subsubsection{Results}
The following theorem is the most general variant of the Conley
conjecture proved to date.

\begin{Theorem}[Conley Conjecture]
\label{thm:CC}
Assume that $M$ is a closed symplectic manifold satisfying one of the
following conditions:
\begin{itemize}
\item[(CY)] $c_1(TM)|_{\pi_2(M)}=0$,
\item[(NM)] $M$ is negative monotone. 
\end{itemize}
Then every Hamiltonian diffeomorphism $\varphi$ of $M$ with finitely
many fixed points has simple periodic orbits of arbitrarily large
period.
\end{Theorem}

As an immediate consequence, every Hamiltonian diffeomorphism of $M$,
whether or not the fixed-point set is finite, satisfying either (CY)
or (NM) has infinitely many simple periodic orbits. In fact, when the
fixed points of $\varphi$ are isolated, one can be even more specific:
if (CY) holds, every sufficiently large prime $p$ occurs as the period
of a simple orbit and, moreover, one can show that there exists a
sequence of integers $l_i\to \infty$ such that all $p^{l_i}$ are
periods of simple orbits.  Consequently, the number of integers less
than $k$ that occur as periods of simple periodic orbits grows at
least as fast as (in fact, faster than) $k/\ln k-C$ for some constant
$C$. This growth lower bound is typical for the Conley conjecture type
results; see also \cite{GGM} and Section \ref{sec:reeb} for the case
of Reeb flows, and \cite{Hi:geod} for the growth of closed geodesics
on $S^2$. (In dimension two, however, stronger growth results have
been established in some cases; see, e.g., \cite{LeC} and \cite[Prop.\
4.13]{Vi:gen} and also \cite{BH, FH, Ke:JMD}.)  When $M$ is negative
monotone, it is only known that there is a sequence of arbitrarily
large primes occurring as simple periods at least when, in addition,
$\varphi$ is assumed to be weakly non-degenerate; see Section
\ref{sec:indices} for the definition.

The (CY) case of the theorem is proved in \cite{He:irr}; see also
\cite{Gi:CC} and, respectively, \cite{GG:gaps} for the proofs when $M$
is symplectically aspherical, and when $M$ is rational and (CY) holds.
The negative monotone case is established in \cite{CGG, GG:nm}. 
For both classes of the ambient manifolds, the
proof of Theorem \ref{thm:CC} amounts to analyzing two cases. When $M$
is CY, the ``non-degenerate case'' of the Conley conjecture is based
on the observation, going back to \cite{SZ}, that unless $\varphi$ has
a fixed point of a particular type called a \emph{symplectically
  degenerate maximum} or an \emph{SDM}, new simple periodic orbits of
high period must be created to generate the Floer homology in degree
$n=\dim M/2$. (For negative monotone manifolds, the argument is more
involved.) In the ``degenerate case'' one shows that the presence of
an SDM fixed point implies the existence of infinitely many periodic
orbits; see \cite{Hi} and also \cite{Gi:CC}.  We outline the proof of
Theorem \ref{thm:CC} for rational CY manifolds in Section
\ref{sec:pf}.

Among closed symplectic manifolds $M$ with $c_1(TM)|_{\pi_2(M)}=0$ are
tori and Calabi--Yau manifolds. In fact, the manifolds meeting this
requirement are more numerous than it might seem. As is proved in
\cite{FP}, for every finitely presented group $G$ there exists a
closed symplectic 6-manifold $M$ with $\pi_1(M)=G$ and $c_1(TM)=0$. A
basic example of a negative monotone symplectic manifold is a smooth
hypersurface of degree $d>n+2$ in $\CP^{n+1}$.  More generally, a
transverse intersection $M$ of $m$ hypersurfaces of degrees
$d_1,\dotsc,d_m$ in $\CP^{n+m}$ is negative monotone iff $d_1+\dotsb
+d_m>n+m+1$; see \cite[p.\ 88]{LM} and also, for $n=4$, \cite[p.\
429--430]{MS}. A complete intersection $M$ is Calabi--Yau when
$d_1+\dotsb+d_m=n+m+1$ and (strictly) monotone when $d_1+\dotsb+
d_m<n+m+1$. Thus ``almost all'' complete intersections are negative
monotone. Note also that the product of a symplectically aspherical
and negative monotone manifolds is again negative monotone.

As has been pointed out in Section \ref{sec:CC-history}, we expect an
analog of the theorem to hold when $N$ is large. (In the (CY) case,
$N=\infty$.) However, at this stage it is only known that the number
of simple periodic orbits is bounded from below by $\lceil N/n \rceil$
when $M^{2n}$ is rational and $2N> 3n$; see \cite[Thm.\ 1.3]{GG:gaps}.

Let us now turn to the question of the generic existence of infinitely
many simple periodic orbits. Conjecturally, for any closed symplectic
manifold $M$, a $C^\infty$-generic Hamiltonian diffeomorphism has
infinitely many simple periodic orbits . This, however, is unknown
(somewhat surprisingly) and appears to be a non-trivial problem. In
all results to date, some assumptions on $M$ are required for the
proof.

\begin{Theorem}[Generic Existence]
\label{thm:generic}
Assume that $M^{2n}$ is a closed symplectic manifold with minimal
Chern number $N$, meeting one of the following requirements:
\begin{itemize}
\item[(i)] $\H_{\textrm{odd}}(M;R)\neq 0$ for some ring $R$,
\item[(ii)] $N\geq n+1$, 
\item[(iii)] $M$ is $\CP^n$ or a complex Grassmannian or a product of
  one of these manifolds with a closed symplectically aspherical
  manifold.

\end{itemize}
Then strongly non-degenerate Hamiltonian diffeomorphisms with
infinitely many simple periodic orbits form a $C^\infty$-residual set
in the space of all $C^\infty$-smooth Hamiltonian diffeomorphisms.
\end{Theorem}

This theorem is proved in \cite{GG:generic}.  In (iii), instead of
explicitly specifying $M$, we could have required that $M$ is monotone
and that there exists $u\in \H_{*<2n}(M)$ with $2n-\deg u<2N$ and
$w\in \H_{*<2n}(M)$ and $\alpha$ in the Novikov ring of $M$ such that
$[M]=(\alpha u)*w$ in the quantum homology. We refer the reader to
\cite{GG:generic} for other examples when this condition is satisfied
and a more detailed discussion.

The proof of the theorem when (i) holds is particularly
simple. Namely, in this case, a non-degenerate Hamiltonian
diffeomorphism $\varphi$ with finitely many periodic orbits must have
a non-hyperbolic periodic orbit. (Indeed, it follows from Floer theory
that $\varphi$ has a non-hyperbolic fixed point or a hyperbolic fixed
point with negative eigenvalues. When $\varphi$ has finitely many
periodic orbits, we can eliminate the latter case by passing to an
iteration of $\varphi$.) To finish the proof it suffices to apply the
Birkhoff--Lewis--Moser theorem, \cite{Mo}. The proofs of the remaining
cases rely on the fact, already mentioned in Section
\ref{sec:CC-history}, that under our assumptions on $M$ the indices
and/or actions of the periodic orbits of $\varphi$ must satisfy
certain resonance relations when $\varphi$ has only finitely many
periodic orbits; see \cite{GG:gaps, GK09}. These resonance relations
can be easily broken by a $C^\infty$-small perturbation of $\varphi$,
and the theorem follows.

It is interesting to look at these results in the context of the
closing lemma, which implies that the existence of a dense set of
periodic orbits is $C^1$-generic for Hamiltonian diffeomorphisms; see
\cite{PR}. Thus, once the $C^\infty$-topology is replaced by the
$C^1$-topology a much stronger result than the generic existence of
infinitely many periodic orbits holds -- the generic dense
existence. However, this is no longer true for the $C^r$-topology with
$r>\dim M$ as the results of M. Herman show (see \cite{He91a, He91b}),
and the above conjecture on the $C^\infty$-generic existence of
infinitely many periodic orbits can be viewed as a hypothetical
variant of a $C^\infty$-closing lemma. (Note also that in the closing
lemma one can require the perturbed diffeomorphism to be
$C^\infty$-smooth, but only $C^1$-close to the original one, as long
as only finitely many periodic orbits are created. It is not 
clear to us whether one can produce infinitely many periodic
orbits by a $C^\infty$-smooth $C^1$-small perturbation.)

\begin{Remark}
  The proofs of Theorem \ref{thm:CC} and \ref{thm:generic} utilize
  Hamiltonian Floer theory. Hence, either $M$ is required in addition
  to be weakly monotone (i.e., $M$ is monotone or $N>n-2$; see
  \cite{HS, MS, Ono:AC, Sa} for more details) or the proofs
  ultimately, although not explicitly, must rely on the machinery of
  multi-valued perturbations and virtual cycles (see \cite{FO, FOOO,
    LT} or, for the polyfold approach, \cite{HWZ:SC, HWZ:poly} and
  references therein). In the latter case, the ground field in the
  Floer homology must have zero characteristic.
\end{Remark}

\section{Outline of the proof of the Conley conjecture}
\label{sec:pf}

\subsection{Preliminaries}
\label{sec:CC-prelim}
In this section we recall, very briefly, several definitions and
results needed for the proof of the (CY) case of Theorem \ref{thm:CC}
and also some terminology used throughout the paper.

\subsubsection{The mean and Conley--Zehnder indices}
\label{sec:indices}
To every continuous path $\Phi\colon [0,\,1]\to\Sp(2n)$ starting at
$\Phi(0)=I$ one can associate the \emph{mean index} $\Delta(\Phi)\in
\R$, a homotopy invariant of the path with fixed end-points; see
\cite{Lo, SZ}.  To give a formal definition, recall first that a map
$\Delta$ from a Lie group to $\R$ is said to be a quasimorphism if it
fails to be a homomorphism only up to a constant, i.e.,
$|\Delta(\Phi\Psi)-\Delta(\Phi)-\Delta(\Psi)|<\const$, where the
constant is independent of $\Phi$ and $\Psi$.  Then one can prove that
there is a unique quasimorphism $\Delta\colon
\widetilde{\Sp}(2n)\to\R$, which is continuous and homogeneous (i.e.,
$\Delta(\Phi^k)=k\Delta(\Phi)$), and satisfies the normalization
condition: $\Delta(\Phi_0)=2$ for $\Phi_0(t)= e^{2\pi i t}\oplus
I_{2n-2}$ with $t\in [0,\,1]$, in the self-explanatory notation; see
\cite{BG}. This quasimorphism is the mean index. (The continuity
requirement holds automatically and is not necessary for the
characterization of $\Delta$, although this is not immediately
obvious. Furthermore, $\Delta$ is also automatically conjugation
invariant, as a consequence of the homogeneity.)

The mean index $\Delta(\Phi)$ measures the total rotation angle of
certain unit eigenvalues of $\Phi(t)$ and can be explicitly defined as
follows. For $A\in\Sp(2)$, set $\rho(A)=e^{i\lambda}\in S^1$ when $A$
is conjugate to the rotation in $\lambda$ counterclockwise,
$\rho(A)=e^{-i\lambda}\in S^1$ when $A$ is conjugate to the rotation
in $\lambda$ clockwise, and $\rho(A)=\pm 1$ when $A$ is hyperbolic
with the sign determined by the sign of the eigenvalues of $A$. Then
$\rho\colon \Sp(2)\to S^1$ is a continuous (but not $C^1$) function,
which is conjugation invariant and equal to $\det$ on $\U(1)$. A
matrix $A\in\Sp(2n)$ with distinct eigenvalues, can be written as the
direct sum of matrices $A_j\in\Sp(2)$ and a matrix with complex
eigenvalues not lying on the unit circle. We set $\rho(A)$ to be the
product of $\rho(A_j)\in S^1$. Again, $\rho$ extends to a continuous
function $\rho\colon \Sp(2n)\to S^1$, which is conjugation invariant
(and hence $\rho(AB)=\rho(BA)$) and restricts to $\det$ on
$\U(n)$; see, e.g., \cite{SZ}. Finally, given a path $\Phi\colon
[0,\,1]\to \Sp(2n)$, there is a continuous function $\lambda(t)$ such
that $\rho(\Phi(t))=e^{i\lambda(t)}$, measuring the total rotation of
the ``preferred'' eigenvalues on the unit circle, and we set
$\Delta(\Phi)=(\lambda(1)-\lambda(0))/2$.

Assume now that the path $\Phi$ is \emph{non-degenerate}, i.e., by
definition, all eigenvalues of the end-point $\Phi(1)$ are different
from one. We denote the set of such matrices in $\Sp(2n)$ by
$\Sp^*(2n)$. It is not hard to see that $\Phi(1)$ can be
connected to a hyperbolic symplectic transformation by a path $\Psi$
lying entirely in $\Sp^*(2n)$. Concatenating this path with $\Phi$, we
obtain a new path $\Phi'$. By definition, the \emph{Conley--Zehnder
  index} $\MUCZ(\Phi)\in\Z$ of $\Phi$ is $\Delta(\Phi')$. One can show that
$\MUCZ(\Phi)$ is well-defined, i.e., independent of
$\Psi$. Furthermore, following \cite{SZ}, let us call $\Phi$
\emph{weakly non-degenerate} if at least one eigenvalue of $\Phi(1)$
is different from one and \emph{totally degenerate} otherwise. The
path is \emph{strongly non-degenerate} if all its ``iterations''
$\Phi^k$ are non-degenerate, i.e., none of the eigenvalues of
$\Phi(1)$ is a root of unity.

The indices $\Delta$ and $\MUCZ$ have the following properties:
\begin{itemize}
\item[\reflb{(CZ1)}{(CZ1)}] $|\Delta(\Phi)-\MUCZ(\tPhi)|\leq n$ for
  every sufficiently small non-degenerate perturbation $\tPhi$ of
  $\Phi$; moreover, the inequality is strict when $\Phi$ is weakly
  non-degenerate.
\item[\reflb{(CZ2)}{(CZ2)}] $\MUCZ(\Phi^k)/k\to \Delta(\Phi)$ as
  $k\to\infty$, when $\Phi$ is strongly non-degenerate; hence, the
  name ``mean index'' for $\Delta$.
\end{itemize}  

Note that with our conventions the Conley--Zehnder index of a path
parametrized by $[0,\,1]$ and generated by a small negative definite
quadratic Hamiltonian on $\R^{2n}$ is $n$.

Let now $M^{2n}$ be a symplectic manifold and $x\colon S^1\to M$ be a
contractible loop.  A \emph{capping} of $x$ is a map $u\colon D^2\to
M$ such that $u|_{S^1}=x$. Two cappings $u$ and $v$ of $x$ are
considered to be equivalent if the integrals of $c_1(TM)$ and of
$\omega$ over the sphere $u\# v$ obtained by attaching $u$ to $v$ is
equal to zero. A capped closed curve $\bar{x}=(x,u)$ is, by
definition, a closed curve $x$ equipped with an equivalence class of a
capping. In what follows, a capping is always indicated by the bar.

For a capped one-periodic (or $k$-periodic) orbit $\bx$ of a
Hamiltonian $H\colon S^1\times M\to \R$, we can view the linearized
flow $d\varphi_H^t|_x$ along $x$ as a path in $\Sp(2n)$ by fixing a
trivialization of $u^*TM$ and restricting it to $x$. With this
convention in mind, the above definitions and constructions apply to
$\bx$ and, in particular, we have the mean index $\Delta(\bx)$ and,
when $x$ is non-degenerate, the Conley--Zehnder index $\MUCZ(\bx)$
defined. These indices are independent of the trivialization of
$u^*TM$, but may depend on the capping.  Furthermore, \ref{(CZ1)} and
\ref{(CZ2)} hold. The difference of the indices of $(x,u)$ and $(x,v)$
is equal to $2\left<c_1(TM),u\# v\right>$. Hence, when $M$ is a
symplectic CY manifold, the indices are independent of the capping and
thus assigned to $x$.  The terminology we introduced for paths in
$\Sp(2n)$ translates word-for-word to periodic orbits and Hamiltonian
diffeomorphisms, cf.\ Section \ref{sec:conventions}.

\subsubsection{Floer homology} 
\label{sec:FH}
In this section, we recall the construction and basic properties of
global, filtered and local Hamiltonian Floer homology on a weakly
monotone, rational symplectic manifold $(M^{2n},\omega)$.

The \emph{action} of a one-periodic Hamiltonian $H$ on a capped
loop $\bar{x}=(x,u)$ is, by definition,
$$
\CA_H(\bar{x})=-\int_u\omega+\int_{S^1} H_t(x(t))\,dt.
$$
The space of capped closed curves is a covering space of the space of
contractible loops, and the critical points of $\CA_H$ on this
covering space are exactly capped one-periodic orbits of $X_H$. The
\emph{action spectrum} $\CS(H)$ of $H$ is the set of critical values
of $\CA_H$. This is a zero measure set; see, e.g., \cite{HZ}.  When M
is rational, $\CS(H)$ is a closed and hence nowhere dense
set. (Otherwise, $\CS(H)$ is everywhere dense.)  These definitions
extend to $k$-periodic orbits and Hamiltonians in the obvious
way. Clearly, the action functional is homogeneous with respect to
iteration:
\begin{equation*}
\CA_{H^{\nat k}}(\bx^k)=k\CA_H(\bx).
\end{equation*}
Here $\bx^k$ stands for the $k$th iteration of the capped orbit $\bx$.

For a Hamiltonian $H\colon S^1\times M\to \R$ and $\varphi=\varphi_H$,
we denote by $\HF_*(\varphi)$ or, when the action filtration is
essential, by $\HF_*^{(a,\,b)}(H)$ the \emph{Floer homology} of $H$,
where $a$ and $b$ are not in $\CS(H)$. We refer the reader to, e.g.,
\cite{MS, Sa} for a detailed construction of the Floer homology and to
\cite{GG:gaps} for a treatment particularly tailored for our
purposes. Here we only mention that, when $H$ is non-degenerate,
$\HF_*^{(a,\,b)}(H)$ is the homology of a complex generated by the
capped one-periodic orbits of $H$ with action in the interval
$(a,\,b)$ and graded by the Conley--Zehnder index. Furthermore,
$\HF_*(\varphi)\cong \H_{*+n}(M)\otimes \Lambda$, where $\Lambda$ is a
suitably defined Novikov ring. As a consequence, $\HF_n(\varphi)\neq
0$ when $M$ is symplectic CY. (For our purposes it is sufficient to
take $\Z_2$ as the ground field.)

When $x$ is an isolated one-periodic orbit of $H$, one can associate
to it the so-called \emph{local Floer homology} $\HF_*(x)$ of
$x$. This is the homology of a complex generated by the orbits $x_i$
which $x$ splits into under a $C^\infty$-small non-degenerate
perturbation. The differential $\p$ is defined similarly to the
standard Floer differential, and to show that $\p^2=0$ it suffices to
prove that the Floer trajectories connecting the orbits $x_i$ cannot
approach the boundary of an isolating neighborhood of $x$. This is an
immediate consequence of \cite[Thm.\ 3]{Fl}; see also \cite{McL} for a
different proof. The resulting homology is well defined, i.e.,
independent of the perturbation. The local Floer homology $\HF_*(x)$
carries only a relative grading. To have a genuine $\Z$-grading it is
enough to fix a trivialization of $TM|_x$. In what follows, such a
trivialization will usually come from a capping of $x$, and we will
then write $\HF_*(\bx)$.  Clearly, the grading is independent of the
capping when $c_1(TM)|_{\pi_2(M)}=0$. Hence, in the symplectic (CY)
case, the local Floer homology is associated to the orbit $x$
itself. With relative grading, the local Floer homology is defined for
the germ of a time-dependent Hamiltonian flow or, when $x$ is treated
as a fixed point, of a Hamiltonian diffeomorphism. The local Floer
homology is invariant under deformations of $H$ as long as $x$ stays
uniformly isolated.


\begin{Example}
\label{ex:lf}
When $x$ is non-degenerate, $\HF_*(\bx)\cong\Z_2$ is concentrated in
degree $\MUCZ(\bx)$. When $x$ is an isolated critical point of an
autonomous $C^2$-small Hamiltonian $F$ (with trivial capping), the
local Floer homology is isomorphic to the local Morse homology
$\HM_{*+n}(F,x)$ of $F$ at $x$ (see \cite{Gi:CC}), also known as
critical modules, which is in turn isomorphic to
$\H_*(\{F<c\}\cup\{x\}, \{F<c\})$, where $F(x)=c$. The isomorphism
$\HF_*(x)\cong\HM_{*+n}(F,x)$ is a local analog of the isomorphism
between the Floer and Morse homology groups of a $C^2$-small
Hamiltonian; see \cite{SZ} and references therein.
\end{Example}

Let us now state three properties of local Floer homology, which are
essential for what follows.

First of all, $\HF_*(\bx)$ is supported in the interval
$[\Delta(\bx)-n,\Delta(\bx)+n]$:
\begin{equation}
\label{eq:supp}
\supp \HF_*(\bx)\subset [\Delta(\bx)-n,\Delta(\bx)+n],
\end{equation}
i.e., the homology vanishes in the degrees outside this
interval. Moreover, when $x$ is weakly non-degenerate, the support
lies in the open interval. These facts readily follow from \ref{(CZ1)}
and the continuity of the mean index.

Secondly, the local Floer homology groups are building blocks for the
ordinary Floer homology. Namely, assume that for $c\in\CS(H)$ there
are only finitely many one-periodic orbits $\bx_i$ with
$\CA_H(\bx_i)=c$. Then all these orbits are isolated and
$$
\HF_*^{(c-\eps,\,c+\eps)}(H)=\bigoplus \HF_*(\bx_i),
$$
when $M$ is rational and $\eps>0$ is sufficiently small. Furthermore,
it is easy to see that, even without the rationality condition,
$\HF_l(\varphi)=0$ when all one-periodic orbits of $H$ are isolated
and have local Floer homology vanishing in degree $l$.

Finally, the local Floer homology enjoys a certain periodicity
property as a function of the iteration order. To be more specific,
let us call a positive integer $k$ an \emph{admissible iteration} of
$x$ if the multiplicity of the generalized eigenvalue one for the
iterated linearized Poincar\'e return map $d\varphi^k|_x$ is equal to
its multiplicity for $d\varphi|_x$. In other words, $k$ is admissible
if and only if it is not divisible by the degrees of roots of unity
among the eigenvalues of $d\varphi|_x$. For instance, when $x$ is
totally degenerate (the only eigenvalue is one) or strongly
non-degenerate (no roots of unity among the eigenvalues), all $k\in
\N$ are admissible. For any $x$, all sufficiently large primes are
admissible. We have

\begin{Theorem}[\cite{GG:gap}]
\label{thm:persist-lf}
Let $\bx$ be a capped isolated one-periodic orbit of a Hamiltonian
$H\colon S^1\times M\to \R$. Then $x^k$ is also an isolated
one-periodic orbit of $H^{\nat k}$ for all admissible $k$, and the
local Floer homology groups of $\bx$ and $\bx^k$ coincide up to a
shift of degree:
\begin{equation*}
\HF_*(\bx^k)=\HF_{*+s_k}(\bx)
\quad\text{for some $s_k$.}
\end{equation*}
Furthermore, $\lim_{k\to \infty}s_k/k=\Delta(\bx)$ and
$s_k=\Delta(\bx)$ for all $k$ when $x$ is totally degenerate.
Moreover, when $\HF_{n+\Delta(\bx)}(\bx)\neq 0$, the orbit $x$ is
totally degenerate.
\end{Theorem}

The first part of this theorem is an analog of the result from
\cite{GM} for Hamiltonian diffeomorphisms. One can replace a capping
of $x$ by a trivialization of $TM|_x$ with the grading and indices now
associated with that trivialization. The theorem is not obvious,
although not particularly difficult. First, note that by using a
variant of the K\"unneth formula and some simple tricks, one can
reduce the problem to the case where $x$ is a totally degenerate
constant orbit with trivial capping.  (Hence, in particular,
$\Delta(\bx)=0$). Then we have the isomorphisms
$\HF_*(\bx)=\HM_{*+n}(F,x)$, where $F\colon M\to\R$, near $x$, is the
generating function of $\varphi$, and
$\HF_*(\bx^k)=\HM_{*+n}(kF,x)=\HM_{*+n}(F,x)$. Thus, in the totally
degenerate case, $s_k=0$, and the theorem follows; see \cite{GG:gap}
for a complete proof. (The fact that $x^k$ is automatically isolated
when $k$ is admissible, reproved in \cite{GG:gap}, has been known for
some time; see \cite{CMPY}.)

As a consequence of Theorem \ref{thm:persist-lf} or of \cite{CMPY},
the iterated orbit $x^k$ is automatically isolated for all $k$ if it
is isolated for some finite collection of iterations $k$ (depending on
the degrees of the roots of unity among the eigenvalues). Furthermore,
it is easy to see that then the map $k\mapsto \HF_*(\bx^k)$ is periodic up to a
shift of grading, and hence the function $k\mapsto \dim\HF_*(\bx^k)$
is bounded.

An isolated orbit $x$ is said to be \emph{homologically non-trivial}
if $\HF_*(x)\neq 0$. (The choice of trivialization along the orbit is
clearly immaterial here.) These are the orbits detected by the
filtered Floer homology. For instance, a non-degenerate orbit is
homologically non-trivial. By Theorem \ref{thm:persist-lf}, an
admissible iteration of a homologically non-trivial orbit is again
homologically non-trivial. It is not known if, in general, an
iteration of a homologically non-trivial orbit can become
homologically trivial while remaining isolated.

We refer the reader to \cite{Gi:CC, GG:gaps, GG:gap} for a further
discussion of local Floer homology.

As we noted in Section \ref{sec:CC-results}, the proof of the general
case of the Conley conjecture for symplectic CY manifolds hinges on
the fact that the presence of an orbit of a particular type, a
\emph{symplectically degenerate maximum} or an \emph{SDM},
automatically implies the existence of infinitely many simple periodic
orbits.  To be more precise, an isolated periodic orbit $x$ is said to be
a symplectically degenerate maximum if $\HF_{n}(x)\neq 0$ and
$\Delta(x)=0$ for some trivialization. This definition makes sense
even for the germs of Hamiltonian flows or Hamiltonian
diffeomorphisms.  An SDM orbit is necessarily totally degenerate by
the ``moreover'' part of \eqref{eq:supp} or Theorem \ref{thm:persist-lf}.
Sometimes it is also convenient to say that an orbit is
an SDM with respect to a particular capping. For instance, a capped
orbit $\bx$ is an SDM if it is an SDM for a trivialization associated
with the capping, i.e., $\HF_{n}(\bx)\neq 0$ and $\Delta(\bx)=0$.

\begin{Example}
\label{ex:sdm} 
Let $H\colon \R^{2n}\to \R$ be an autonomous Hamiltonian with an
isolated critical point at $x=0$. Assume furthermore that $x$ is a
local maximum and that all eigenvalues (in the sense of, e.g.,
\cite[App.\ 6]{Ar}) of the Hessian $d^2H(x)$ are equal to zero. Then
$x$ (with constant trivialization or, equivalently, trivial capping)
is an SDM of $H$. For instance, the origin in $\R^2$ is an SDM for
$H(p,q)=p^4+q^4$ or $H(p,q)=p^2+q^4$, but not for $H(p,q)=ap^2+bq^2$
for any $a\neq 0$ and $b\neq 0$.
\end{Example}

\begin{Remark}
  There are several other ways define an SDM. The following conditions
  are equivalent (see \cite[Prop.\ 5.1]{GG:gap}):
\begin{itemize}

\item the orbit $\bx$ is a symplectically degenerate maximum of $H$;

\item $\HF_n(\bx^{k_i})\neq 0$ for some sequence of admissible
  iterations $k_i\to\infty$;

\item the orbit $x$ is totally degenerate, $\HF_n(\bx)\neq 0$ and
  $\HF_n(\bx^k)\neq 0$ for at least one admissible iteration $k\geq
  n+1$.

\end{itemize}
\end{Remark}

\subsection{The non-degenerate case of the Conley conjecture} 
\label{sec:CC-pf}
The following proposition settling, in particular, the non-degenerate
case of the Conley conjecture for symplectic CY manifolds is a
refinement of the main result from \cite{SZ}. It is proved in
\cite{Gi:CC, GG:gaps}, although the argument given below is somewhat
different from the original proof.

\begin{Proposition}
\label{prop:CC-nd}
Assume that $c_1(TM)|_{\pi_2(M)}=0$ and that $\varphi=\varphi_H$ has
finitely many fixed points and none of these points is an SDM.  (This
is the case when, e.g., $\varphi$ is weakly non-degenerate.) Then
$\varphi$ has a simple periodic orbit of period $k$ for every
sufficiently large prime~$k$.
\end{Proposition}

The key to the proof is the fact that $\HF_n(\varphi^k)\neq 0$ for all
$k$ and that, even when $\omega|_{\pi_2(M)}\neq 0$, the condition
$c_1(TM)|_{\pi_2(M)}=0$ guarantees that all recappings of every orbit
have the same mean index and the same (graded) local Floer homology.

\begin{proof}
  First, note that when $k$ is prime every $k$-periodic orbit is
  either simple or the $k$-th iteration of a fixed point. For every
  isolated fixed point $x$, we have three mutually exclusive
  possibilities:
\begin{itemize}
\item $\Delta(x)\neq 0$,
\item $\Delta(x)=0$ and $\HF_n(x)=0$,
\item $\Delta(x)=0$ but $\HF_n(x)\neq 0$.
\end{itemize}
Here we are using the fact that $M$ is CY, and hence the indices are
independent of the capping. The last case, where $x$ is an SDM, is
rulled out by the assumptions of the proposition.

In the first case, $\HF_n(x^k)=0$ when $k|\Delta(x)|>2n$, and hence
$x$ cannot contribute to $\HF_n(\varphi^k)$ when $k$ is large. In the
second case, $\HF_n(x^k)=\HF_n(x)=0$ for all admissible iterations by
Theorem \ref{thm:persist-lf}.  In particular, $x$ again cannot
contribute to $\HF_n(\varphi^k)$ for all large primes $k$. It follows
that, under the assumptions of the proposition, $\HF_n(\varphi^k)=0$
for all large primes $k$ unless $\varphi$ has a simple periodic orbit
of period~$k$.
\end{proof}

\begin{Remark}
  Although this argument relies on Theorem \ref{thm:persist-lf} which
  is not entirely trivial, a slightly different logical organization
  of the proof would enable one to utilize on a much simpler of
  version of the theorem; see \cite{Gi:CC, GG:gaps}.
\end{Remark}

With Proposition \ref{prop:CC-nd} established, it remains to deal with
the degenerate case of the Conley conjecture, i.e., the case where
$\varphi$ has an SDM. We do this in the next section; see Theorem
\ref{thm:sdm-conley}.

\begin{Remark}
  When $M$ is negative monotone and $\varphi$ has an SDM fixed point,
  the degenerate case of Theorem \ref{thm:CC} follows from Theorem
  \ref{thm:sdm-conley}, just as for the CY manifolds. However, the
  non-degenerate case requires a totally new proof. The argument
  relies on the sub-additivity property of spectral invariants; see
  \cite{CGG,GG:nm} for more details.
\end{Remark}

\subsection{Symplectically degenerate maxima}
\label{sec:sdm}
In this section, we show that a Hamiltonian diffeomorphism with an SDM
fixed point has infinitely many simple periodic orbits. We assume that
$M$ is rational as in \cite{GG:gaps}. The case of irrational CY
manifolds is treated in \cite{He:irr}.

\begin{Theorem}[\cite{GG:gaps}]
\label{thm:sdm-conley}
Let $\varphi=\varphi_H$ be a Hamiltonian diffeomorphism of a closed
rational symplectic manifold $M$, generated by a one-periodic
Hamiltonian $H$.  Assume that some iteration $\varphi^{k_0}$ has
finitely many $k_0$-periodic orbits and one of them, $\bx$, is an SDM.
\begin{itemize}

\item[(i)] Then $\varphi$ has infinitely many simple periodic orbits.

\item[(ii)] If, in addition, $k_0=1$ and $\omega|_{\pi_2(M)}=0$ or
  $c_1(M)|_{\pi_2(M)}=0$, then $\varphi$ has simple periodic orbits of
  arbitrarily large prime period.

\end{itemize}
\end{Theorem}

This theorem is in turn a consequence of the following result.

\begin{Theorem}[\cite{GG:gaps}]
\label{thm:sdm}
Assume that $(M^{2n},\omega)$ is closed and rational, and let $\bx$ be
an SDM of $H$. Set $c=\A_H(\bx)$.  Then for every sufficiently small
$\eps>0$ there exists $k_\eps$ such that
\begin{equation}
\label{eq:main}
\HF^{(kc+\delta_k,\,kc+\eps)}_{n+1}\big(H^{\nat k}\big)\neq 0\text{ for
  all $k>k_\eps$ and some $\delta_k$ with $0<\delta_k<\eps$.}
\end{equation}
\end{Theorem}

For instance, to prove case (ii) of Theorem \ref{thm:sdm-conley} when
$M$ is CY it suffices to observe that no $k$th iteration of a fixed
point can contribute to the Floer homology in degree $n+1$ for any
action interval when $k$ is sufficiently large prime and $\varphi$ has
finitely many fixed points.  When $\omega|_{\pi_2(M)}=0$, the argument
is similar, but now the action filtration is used in place of the
degree. The proof of case (i) is more involved; see \cite[Sect.\
3]{GG:gaps} where some more general results are also established.

\begin{proof}[Outline of the proof of Theorem \ref{thm:sdm}]
  Composing if necessary $\varphi^t_H$ with a loop of Hamiltonian
  diffeomorphisms, we can easily reduce the problem to the case where
  $\bx$ is a constant one-periodic orbit with trivial capping; see
  \cite[Prop.\ 2.9 and 2.10]{GG:gaps}. Henceforth, we write $x$ rather
  than $\bx$ and assume that $dH_t(x)=0$ for all $t$.

  The key to the proof is the following geometrical characterization
  of SDMs:

\begin{Lemma}[\cite{Gi:CC, Hi}]
\label{lemma:sdm}
Let $x$ be an isolated constant one-periodic orbit for a germ of a
time-dependent Hamiltonian flow $\varphi_H^t$. Assume that $x$ (with
constant trivialization) is an SDM. Then there exists a germ of a
time-dependent Hamiltonian flow $\varphi_K^t$ near $x$ such that the
two flows generate the same time-one map, i.e., $\varphi_K=\varphi_H$,
and $K_t$ has a strict local maximum at $x$ for every
$t$. Furthermore, one can ensure that the Hessian $d^2K_t(x)$ is
arbitrarily small. In other words, for every $\eta>0$ one can find
such a Hamiltonian $K_t$ with $\|d^2K_t(x)\|<\eta$.
\end{Lemma}

\begin{Remark} Strictly speaking, contrary to what is stated in
  \cite[Prop.\ 5.2]{GG:gap} and \cite[Rmk.\ 5.9]{GG:gaps}, this lemma
  is not quite a characterization of SDMs in the sense that it is not
  clear if every $x$ for which such Hamiltonians $K_t$ exist is
  necessarily an SDM. However, in fact, $K_t$ can be taken to meet an
  additional requirement ensuring, in essence, that the $t$ dependence
  of $K_t$ is minor. With this condition, introduced in \cite[Lemma
  4]{Hi} as that $K$ is relatively autonomous (see also \cite[Sect.\ 5
  and 6]{Gi:CC}), the lemma gives a necessary and sufficient condition
  for an SDM.
\end{Remark}

\begin{proof}[Outline of the proof of Lemma \ref{lemma:sdm}]
  The proof of Lemma \ref{lemma:sdm} is rather technical, but the idea
  of the proof is quite simple. Set $\varphi=\varphi_H$. First, note
  that all eigenvalues of $d\varphi|_{x}\colon T_xM\to T_xM$ are equal
  to one since $x$ is totally degenerate. Thus by applying a
  symplectic linear change of coordinates we can bring $d\varphi|_{x}$
  arbitrarily close to the identity. Then $\varphi$ is also
  $C^1$-close to $\id$ near $x$. Let us identify $(M\times M,
  \omega\oplus (-\omega))$ near $(x,x)$ with a neighborhood of the
  zero section in $T^*M$, and hence the graph of $\varphi$ with the
  graph of $dF$ for a germ of a smooth function $F$ near $x$. The
  function $F$ is a generating function of $\varphi$. Clearly, $x$ is
  an isolated critical point of $F$ and $d^2F(x)=O(\| d\varphi|_{x} -
  I \|)$.

  Furthermore, similarly to Example \ref{ex:lf}, we have an
  isomorphism
$$
\HF_*(x)=\HM_{*+n}(F,x),
$$
and thus $\HM_{2n}(F,x)\neq 0$. It is not hard to show that an
isolated critical point $x$ of a function $F$ is a local maximum if
and only if $\HM_{2n}(F,x)\neq 0$. The generating function $F$ is not
quite a Hamiltonian generating $\varphi$, but it is not hard to turn
$F$ into such a Hamiltonian $K_t$ and check that $K_t$ inherits the
properties of $F$.
\end{proof}

Returning to the proof of Theorem \ref{thm:sdm}, we apply the lemma to
the SDM orbit $x$ and observe that the local loop
$\varphi_H^t\circ(\varphi_K^t)^{-1}$ has zero Maslov index and hence
is contractible. It is not hard to show that every local contractible
loop extends to a global contractible loop; see \cite[Lemma
2.8]{Gi:CC}. In other words, we can extend the Hamiltonian $K_t$ from
Lemma \ref{lemma:sdm} to a global Hamiltonian such that
$\varphi_K=\varphi$, not only near $x$ but on the entire manifold $M$.

With this in mind, let us reset the notation. Replacing $H$ by $K$ but
retaining the original notation, we can say that for every $\eta>0$
there exists a Hamiltonian $H$ such that
\begin{itemize}
\item $\varphi_H=\varphi$;
\item $x$ is a constant periodic orbit of $H$, and $H_t$ has an
  isolated local maximum at $x$ for all $t$;
 \item $\|d^2H_t(x)\|<\eta$ for all $t$. 
\end{itemize}
Furthermore, we can always assume that all such Hamiltonians $H$ are
related to each other and to the original Hamiltonian via global loops
with zero action and zero Maslov index. Thus, in particular,
$c=\CA_H(x)=H(x)$ is independent of the choice of $H$ above, and all
Hamiltonians have the same filtered Floer homology. Therefore, it is
sufficient to prove the theorem for any of these Hamiltonians $H$ with
arbitrarily small Hessian $d^2H(x)$.

To avoid technical difficulties and illuminate the idea of the proof,
let us asume that $d^2H(x)=0$ and, of course, that $H_t$ has, as
above, a strict local maximum at $x$ for all $t$. This case, roughly
speaking, corresponds to an SDM $x$ with $d\varphi|_x=I$.
 
To prove \eqref{eq:main} for a given $\eps>0$, we will use the
standard squeezing argument, i.e., we will bound $H$ from above and
below by two autonomous Hamiltonians $H_\pm$ as in Fig.\
\ref{fig:functions} and calculate the Floer homology of $kH_\pm$.

\begin{figure}[h!]
\begin{center} 
\def\svgwidth{0.7\columnwidth}
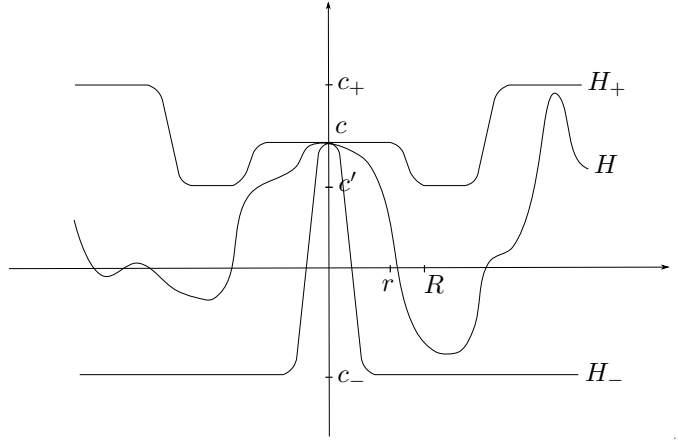
\caption{The functions $H_{\pm}$}
\label{fig:functions}
\end{center}
\end{figure}

In a Darboux neighborhood $U$ of $x$, the Hamiltonians $H_\pm$ are
rotationally symmetric. The Hamiltonian $H_+$ is constant and equal to
$c$ near $x$ on a ball of radius $r$ and then sharply decreases to
some $c'$ which is close to $c$ and attained on the sphere of radius
$R$.  Then, after staying constant on a spherical shell, $H_+$
increases to some value $c_+$, to accommodate $H$, and becomes
constant. The radii $r<R$ depend on $\eps$; namely, we require that
$\pi R^2<\eps$.  The Hamiltonian $H_-$ is a bump function decreasing
from its maximum $c$ at $x$ to a large negative value $c_-$. Thus we
have
$$
H_+\geq H\geq H_-.
$$
We require $H_-$ to have a strict maximum at $x$ with
$d^2H_-(x)=0$. Then the local Floer homology of $H_-^{\nat k}$ is
equal to $\Z_2$ and concentrated in degree $n$, i.e., $x$ is also an
SDM for $H_-$ and all its iterations.

Now, for any $a<b$ outside $\CS(H^{\nat k})$ and $\CS(H^{\nat
  k}_\pm)$, we have the maps
$$
\HF^{(a,\,b)}_{*}\big(H^{\nat k}_+\big)\to
\HF^{(a,\,b)}_{*}\big(H^{\nat k}\big) \to
\HF^{(a,\,b)}_{*}\big(H^{\nat k}_-\big)
$$
induced by monotone homotopies, where $\HF^{(a,\,b)}_{*}\big(H^{\nat
  k}_\pm\big)=\HF^{(a,\,b)}_{*}\big(kH_\pm\big)$ since $H_{\pm}$ are
autonomous Hamiltonians. Therefore, it is be sufficient to prove that
the map
\begin{equation}
\label{eq:iso-main}
 \HF^{(kc+\delta,\,kc+\eps)}_{n+1}\big(kH_+\big)
 \to \HF^{(kc+\delta,\,kc+\eps)}_{n+1}\big(kH_-\big)
\end{equation}
is non-zero for some $\delta$ in the range $(0,\,\eps)$.

To this end, let us assume first that $H_\pm$, as above, are functions
on $\R^{2n}$ constant outside a neighborhood of $x=0$. The filtered
Floer homology of $H_\pm$ is still defined for any interval $(a,\,b)$
not containing $c_\pm$. Moreover, a decreasing homotopy $H^s$ from
$H_+$ to $H_-$ through functions constant outside a compact set
induces a map in the Floer homology even when the value of $H^s$ at
infinity passes through $(a,\,b)$. Then we have an isomorphism
 \begin{equation}
\label{eq:isom}
\Z_2\cong \HF^{(kc+\delta,\,kc+\eps)}_{n+1}\big(kH_+\big)
\stackrel{\cong}{\longrightarrow} \HF^{(kc+\delta,\,kc+\eps)}_{n+1}\big(kH_-\big)\cong\Z_2
\end{equation}
when $k$ is sufficiently large and $\delta>0$ is sufficiently
small. Namely, $k$ is so large that $k(c-c')>\pi R^2$. This is the
origin of the requirement $k>k_\eps$. Then the homology of $kH_\pm$ is
generated by the periodic orbit closest to $x$, and $\delta$ is chosen
so that $kc+\delta$ is smaller than the action of $kH_-$ on this
orbit.

The isomorphism \eqref{eq:isom} is established by a straightforward
analysis of periodic orbits and easily follows from the calculation
carried out already in \cite{GG:capacity}. It is based on two facts:
that
$$
\Z_2\cong \HF^{(kc-\delta,\,kc+\delta)}_{n}\big(kH_+\big)
\stackrel{\cong}{\longrightarrow}
\HF^{(kc-\delta,\,kc+\delta)}_{n}\big(kH_-\big)\cong\Z_2
$$
is an isomorphism when $\delta>0$ is small and that
$\HF^{(a,\,b)}_{n+1}\big(kH_\pm\big)=0$ for every sufficiently large
interval $(a,\,b)$ containing $kc$ and contained in $(kc_-,\,kc_+)$.

It remains to transplant this calculation from $\R^{2n}$ to a closed
manifold $M$. The key to this is the fact that the action interval in
question is sufficiently small. This enables one to localize a
calculation of filtered Floer homology by essentially turning action
localization to spatial localization. A general framework for this
process, developed in \cite{GG:gaps}, is as follows. Let $S$ be a
shell in $M$, i.e., a region between two hypersurfaces and bounding a
contractible domain $V$ in $M$.  (To be more precise, $V$ is bounded
in $M$ by a connected component of $\p S$ and $S\cap V=\emptyset$. The
contractibility assumption can be significantly relaxed.) Furthermore,
let $F$ be a Hamiltonian which we require to be constant on the
shell. For any interval $I=(\alpha,\,\beta)$ not containing $F|_S$,
consider the subspace of the Floer complex generated by the orbits of
$F$ in $V$ with cappings also contained in $V$. If necessary, we
perturb $F$ in $V$ to make sure that the orbits with action in $I$ are
non-degenerate.  Then there exists a constant $\eps(S)>0$ such that,
when $|I|=\beta-\alpha<\eps(S)$, this subspace of the Floer complex is
actually a subcomplex and, moreover, a direct summand. (This is an
immediate consequence of the fact that a holomorphic curve crossing
$S$ must have energy bounded away from zero by some constant
$\eps(S)$.)  Furthermore, continuation maps respect this decomposition
as long as the Hamiltonians remain constant on $S$. (However, the
value of the Hamiltonians on $S$ can enter the interval $I$ during the
homotopy.) Let us denote the resulting Floer homology by
$\HF_*^{I}(F;V)$.

We apply this construction to $H_\pm$ with $S$ being the spherical
shell where $H_+=c'$ and $V$ being the ball of radius $R$ in $U$
enclosed by this shell. (Hence we also need $H_-$ to be constant
outside $V$.)  As a consequence, \eqref{eq:isom} turns into an
isomorphism
$$
\Z_2\cong \HF^{(kc+\delta,\,kc+\eps)}_{n+1}\big(kH_+; V\big)
\stackrel{\cong}{\longrightarrow}
\HF^{(kc+\delta,\,kc+\eps)}_{n+1}\big(kH_-; V\big)\cong\Z_2,
$$
entering the map \eqref{eq:iso-main} as a direct summand. Hence
\eqref{eq:iso-main} is also non-zero.

The general case where we only have $\|d^2H(x)\|<\eta$ is handled in a
similar way, but the construction of $H_\pm$ is considerably more
involved and the choice of the modified Hamiltonians $H$ requires more
attention; see \cite{Gi:CC,GG:gaps}. \end{proof}

Interestingly, no other proof of the Conley conjecture is known for
general symplectic manifolds. A more conceptual or just plain
different argument may shed new light on the nature of the phenomena
considered here and is likely to have other applications. (For the
torus, a different proof is given in the original work \cite{Hi} and
then in \cite{Ma:SDM}. However, it is not clear to us how to translate
that proof to symplectic topological language.)

\begin{Remark}
  The part of the proof that does not go through when $M$ is
  irrational is the last step, the localization. The difficulty is
  that the action spectrum is dense in this case, and necessarily some
  of the recappings of degenerate trivial orbits of $F$ in $S$ have
  actions in $I$. Thus it is not obvious how to define the Floer
  homology localized in $V$. This problem is circumvented in
  \cite{He:irr} by considering the Hamiltonians which have a slight
  slope in $S$ rather than being constant. With this modification, the
  localization procedure goes through, although the underlying reason
  for the localization is now different; see \cite{He:irr, Us:tgf}.
\end{Remark}

\section{Reeb flows}
\label{sec:reeb}
\subsection{General discussion}
\label{sec:reeb-general}
The collection of all closed symplectic manifolds breaks down into two
classes: those for which the Conley conjecture holds and those for
which the Conley conjecture fails. Of course, the non-trivial
assertion is then that, as we have seen, the former class is non-empty
and even quite large. The situation with closed contact manifolds is
more involved even if we leave aside such fundamental questions as the
Weinstein conjecture and furthermore focus exclusively on the contact
homological properties of the manifold.

First of all, there is a class of contact manifolds for which every
Reeb flow has infinitely many simple closed orbits because the rank of
the contact or symplectic homology grows as a function of the index or
of some other parameter related to the order of iteration. This
phenomenon, studied in \cite{HM, McL}, generalizes and is inspired by
the results in \cite{GM} establishing the existence of infinitely many
closed geodesics for manifolds whose free loop space homology grows.
(A technical but important fact closely related to Theorem
\ref{thm:persist-lf} and underpinning the proof is that the iterates
of a given orbit can make only bounded contributions to the homology;
see \cite{GG:gap, GM, HM, McL} for various incarnations of this
result.)  By \cite{VPS} and \cite{AS, SW, Vi:f}, among contact
manifolds in this class are the unit cotangent bundles $ST^*M$
whenever $\pi_1(M)=0$ and the algebra $\H^*(M;\Q)$ is not generated by
one element, and some others, \cite{HM, McL}. As is already pointed
out in Section \ref{sec:CC-history}, this homologically forced
existence of infinitely many Reeb orbits has very different nature
from the Hamiltonian Conley conjecture where there is no homological
growth.

Then there are contact manifolds admitting Reeb flows with finitely
many closed orbits. Among these are, of course, the standard contact
spheres and, more generally, the pre-quantization circle bundles over
symplectic manifolds admitting torus actions with isolated fixed
points; see \cite[Example 1.13]{Gu:pr}. Note that the class of such
pre-quantization circle bundles includes the Katok--Ziller flows,
i.e., Finsler metrics with finitely many closed geodesics on $S^n$ and
on some other manifolds; see \cite{Ka} for the original construction
and also \cite{Zi}. Another important group of examples also
containing the standard contact spheres arises from contact toric
manifolds; see \cite{AM}. These two classes (pre-quantization circle
bundles and contact toric manifolds) overlap, but do not entirely
coincide. Although this is not obvious, Reeb flows with finitely many
periodic orbits may have non-trivial dynamics, e.g., be ergodic; see
\cite{Ka}.

Finally, as is shown in \cite{GGM}, there is a \emph{non-empty} class
of contact manifolds for which every Reeb flow (meeting certain
natural index conditions) has infinitely many simple closed orbits,
although there is no obvious homological growth -- the rank of the
relevant contact homology remains bounded. One can expect this class
to be quite large, but at this point such unconditional existence of
infinitely many closed Reeb orbits has only been proved for the
pre-quantization circle bundles of certain aspherical manifolds; see
Theorem \ref{thm:CCC}. The proof of this theorem is quite similar to
its Hamiltonian counterpart.

This picture is, of course, oversimplified and not even close to
covering the entire range of possibilities, even on the homological
level. For instance, hypothetically, the Reeb flows for overtwisted
contact structures have infinitely many simple closed orbits, but
where should one place such contact structures in our
``classification''? (See \cite{El, Yau} for a proof of the existence
of one closed orbit in this case.)

One application of Theorem \ref{thm:CCC} is the existence of
infinitely many simple periodic orbits for all low energy levels of
twisted geodesic flows on surfaces of genus $g\geq 2$ with
non-vanishing magnetic field; see Section \ref{sec:magnetic}.

Just as in the Hamiltonian setting, the mean indices or the actions
and the mean indices of simple periodic orbits of Reeb flows must, in
many instances, satisfy certain resonance relations when the number of
closed orbits is finite. The mean index resonance relations for the
standard contact sphere were discovered by Viterbo in \cite{Vi:res},
and the Morse--Bott case for geodesic flows was considered in
\cite{Ra:ind}. Viterbo's resonance relations were generalized to
non-degenerate Reeb flows on a broad class of contact manifolds in
\cite{GK09}. These resonance relations resemble the equality between
two expressions for the Euler characteristic of a closed manifold: the
homological one and the one using indices of zeroes of a vector
field. The role of the homological expression is now taken by the mean
Euler characteristic of the contact homology of the manifold,
introduced in \cite{VK}, and the sum of the indices is replaced by the
sum of certain local invariants of simple closed orbits. The
degenerate case of the generalized Viterbo resonance relations was
studied in \cite{GGo, HM, LLW} and the Morse--Bott case in
\cite{Es}. There are also variants of resonance relations involving
both the actions and the mean indices; see \cite{Gu:pr} and also
\cite{Ek, EH}.

Leaving aside the exact form of the resonance relations, we only
mention here some of their applications. The first one, in dynamics,
is a contact analog of Theorem \ref{thm:generic}: the generic
existence of infinitely many simple 
closed orbits for a large class of Reeb flows; see \cite{GG:generic}
and also \cite{Ek, Ra:geod, Hi:eq} for related earlier
results. Another application, also in dynamics, is to the proof of the
existence of at least two simple closed Reeb orbits on the standard
contact $S^3$. This result is further discussed in the next section;
see Theorem \ref{thm:S3}. (We refer the reader to \cite{Gu:pr} for
some other applications in dynamics.)  Finally, on the topological
side, the resonance relations can be used to calculate the mean Euler
characteristic, \cite{Es}.

\subsection{Contact Conley conjecture}
\label{sec:CCC}
Consider a closed symplectic manifold $(M,\omega)$ such that the form
$\omega$ or, to be more precise, its cohomology class $[\omega]$ is
\emph{integral}, i.e., $[\omega]\in \H^2(M;\Z)/\Tor$. Let $\pi\colon
P\to M$ be an $S^1$-bundle over $M$ with the first Chern class
$-[\omega]$.  The bundle $P$ admits an $S^1$-invariant 1-form
$\alpha_0$ such that $d\alpha_0=\pi^*\omega$ and $\alpha_0(R_0)=1$,
where $R_0$ is the vector field generating the $S^1$-action on $P$. In
other words, when we set $S^1=\R/\Z$ and identify the Lie algebra of
$S^1$ with $\R$, the form $\alpha_0$ is a connection form on $P$ with
curvature $\omega$. (Note our sign convention.)

Clearly, $\alpha_0$ is a contact form with Reeb vector field $R_0$,
and the connection distribution $\xi=\ker\alpha_0$ is a contact
structure on $P$. Up to a gauge transformation, $\xi$ is independent
of the choice of $\alpha_0$. The circle bundle $P$ equipped with this
contact structure is usually referred to as a \emph{pre-quantization
  circle bundle} or a Boothby--Wang bundle.  Also, recall that a
degree two (real) cohomology class on $P$ is said to be
\emph{atoroidal} if its integral over any smooth map $\T^2\to P$ is
zero. (Such a class is necessarily aspherical.)  Finally, in what
follows, we will denote by $\ff$ the free homotopy class of the fiber
of $\pi$.

The main tool used in the proof of Theorem \ref{thm:CCC} stated below
is the cylindrical contact homology. As is well known, to have this
homology defined for a contact form $\alpha$ on any closed contact
manifold $P$ one has to impose certain additional requirements on the
closed Reeb orbits of $\alpha$. (See \cite{Bo, SFT} and references
therein for the definition and a detailed discussion of contact
homology.) Namely, following \cite{GGM}, we say that a non-degenerate
contact form $\alpha$ is \emph{index--admissible} if its Reeb flow has
no contractible closed orbits with Conley--Zehnder index $2-n$ or
$2-n\pm 1$, where $\dim P=2n+1$. In general, $\alpha$ or its Reeb flow
is index--admissible when there exists a sequence of non-degenerate
index--admissible forms $C^1$-converging to~$\alpha$.

This requirement is usually satisfied when $(P,\alpha)$ has some
geometrical convexity properties. For instance, the Reeb flow on a
strictly convex hypersurface in $\R^{2m}$ is index--admissible,
\cite{HWZ:convex}. Likewise, as is observed in \cite{Be}, the twisted
geodesic flow on a low energy level for a symplectic magnetic field on
a surface of genus $g\geq 2$ is index--admissible; see Section
\ref{sec:magnetic} for more details. Finally, let us call a closed
Reeb orbit non-degenerate (or weakly non-degenerate, SDM, etc.)  if
its Poincar\'e return map is non-degenerate (or, respectively, weakly
non-degenerate, SDM, etc.), cf.\ Section \ref{sec:conventions}.

\begin{Theorem}[Contact Conley Conjecture, \cite{GGM}]
\label{thm:CCC}
Assume that 
\begin{itemize}
\item[(i)] $M$ is aspherical, i.e., $\pi_r(M)=0$ for all $r\geq 2$,
  and
\item[(ii)] $c_1(\xi)\in \H^2(P;\R)$ is atoroidal.
\end{itemize}
Let $\alpha$ be an index--admissible contact form on the
pre-quantization bundle $P$ over $M$, supporting
$\xi$. Then the Reeb flow of $\alpha$ has infinitely many simple
closed orbits with contractible projections to $M$. Assume furthermore
that the Reeb flow has finitely many periodic orbits in the free
homotopy class $\ff$ of the fiber and that these orbits are weakly
non-degenerate. Then for every sufficiently large prime $k$ the Reeb
flow of $\alpha$ has a simple closed orbit in the class $\ff^k$, and
all classes $\ff^k$ are distinct.
\end{Theorem}

It follows from Theorem \ref{thm:CCC} and the discussion below that,
when the Reeb flow of $\alpha$ is weakly non-degenerate, the number of
simple periodic orbits of the Reeb flow of $\alpha$ with period (or
equivalently action) less than $a>0$ is bounded from below by $C_0
\cdot a/\ln a - C_1$, where $C_0=\inf \alpha(R_0)$ and $C_1$ depends
only on $\alpha$. As mentioned in Section \ref{sec:CC-results}, this
is a typical growth lower bound in the Conley conjecture type results.
Note also that the weak non-degeneracy requirement here plays a
technical role and probably can be eliminated.

The key to the proof of Theorem \ref{thm:CCC} is the observation that,
as a consequence of (i), all free homotopy classes $\ff^k$, $k\in\N$,
are distinct and hence give rise to an $\N$-grading of the cylindrical
contact homology of $(P,\alpha)$. (In fact, it would be sufficient to
assume that $[\omega]$ is aspherical and $\pi_1(M)$ is torsion free;
both of these requirements follow from (i).) This grading plays
essentially the same role as the order of iteration in the Hamiltonian
Conley conjecture. With this observation in mind, the proof of the
weakly non-degenerate case is quite similar to its Hamiltonian
counterpart. (Condition (ii) is purely technical and most likely can
be dropped.)

To complete the proof, one then has to deal with the case where the
Reeb flow of $\alpha$ has a simple SDM orbit, i.e., a simple isolated
orbit with an SDM Poincar\'e return map. This is also done similarly
to the Hamiltonian case, but there are some nuances.

Consider a closed contact manifold $(P^{2n+1},\ker\alpha)$ with a
strong symplectic filling $(W,\omega)$, i.e., $W$ is a compact
symplectic manifold such that $P=\p W$ with $\omega|_P=d\alpha$ and a
natural orientation compatibility condition is satisfied. Let $\fc$ be
a free homotopy class of loops in $W$.

\begin{Theorem}[\cite{GHHM,GGM}]
\label{thm:contact-sdm}
Assume that the Reeb flow of $\alpha$ has a simple closed SDM orbit in
the class $\fc$ and one of the following requirements is met:
\begin{itemize}
\item $W$ is symplectically aspherical and $\fc=1$, or
\item $\omega$ is exact and $c_1(TW)=0$ in $H^2(W;\Z)$.
\end{itemize}
Then the Reeb flow of $\alpha$ has infinitely many simple periodic
orbits.
\end{Theorem}

This result is a contact analog of Theorem \ref{thm:sdm-conley}.
Theorem \ref{thm:CCC} readily follows from the first case of Theorem
\ref{thm:contact-sdm} where we take the pre-quantization disk bundle
over $M$ as $W$. (Here we only point out that $\pi_2(W)=\pi_2(M)=0$
since $M$ is aspherical and refer the reader to \cite{GGM} for more
details.)

The proof of Theorem \ref{thm:contact-sdm} uses the filtered
linearized contact homology. To be more specific, denote by
$\HC_{*}^{(a,\,b)}(\alpha;W,\fc^k)$ the linearized contact homology of
$(P,\alpha)$ with respect to the filling $(W,\omega)$ for the action
interval $(a,\, b)$ and the free homotopy class $\fc^k$, graded by the
Conley--Zehnder index.  Set $\Delta=\Delta(x)$ and $c=\A(x)$ where $x$
is the SDM orbit from the theorem.  Similarly to the Hamiltonian case
(cf.\ Theorem \ref{thm:sdm}), one first shows that, under the
hypotheses of the theorem, for any $\eps>0$ there exists $k_\eps\in
\N$ such that
\begin{equation}
\label{eq:CH-sdm}
\HC_{k\Delta+n+1}^{(kc+\delta_k,\,kc+\eps)}(\alpha;W,\fc^k)\neq 0\textrm{ for
  all } k>k_\eps \textrm{ and some } \delta_k<\eps.
\end{equation}
Theorem \ref{thm:contact-sdm} is a consequence of \eqref{eq:CH-sdm}
(see \cite{GGM}), although the argument is less obvious than its
Hamiltonian counterpart for, say, symplectic CY manifolds.

The proof of \eqref{eq:CH-sdm} given in \cite{GHHM} follows the same
path as the proof Theorem \ref{thm:sdm}. Namely, we squeeze the form
$\alpha$ between two contact forms $\alpha_\pm$ constructed using the
Hamiltonians $H_\pm$ near the SDM orbit, calculate the relevant
contact homology for $\alpha_\pm$ (or rather a direct summand in it),
and show that the map in the contact homology induced by the cobordism
from $\alpha_+$ to $\alpha_-$ is non-zero. This map factors through
$\HC_{k\Delta+n+1}^{(kc+\delta_k,\,kc+\eps)}(\alpha;W,\fc^k)$, and
hence this group is also non-trivial.

Note that Theorem \ref{thm:contact-sdm} as stated, without further
assumptions on $\fc$, affords no control on the free homotopy classes
of the simple orbits or their growth rate.  A related point is that,
as of this writing, there seems to be no satisfactory version of
Theorem \ref{thm:contact-sdm} which would not rely on the existence of
the filling $W$. The difficulty is that without a filling one is
forced to work with cylindrical contact homology to prove a variant of
\eqref{eq:CH-sdm}, but then it is not clear if the forms $\alpha_\pm$
can be made index--admissible without additional assumptions on
$\alpha$ along the lines of index--positivity. Such a filling--free
version of the theorem would, for instance, enable one to eliminate
the weak non-degeneracy assumption in the growth assertion in Theorem
\ref{thm:CCC}. Another serious limitation of Theorem
\ref{thm:contact-sdm} is that the SDM orbit is required to be
simple. This condition, which is quite restrictive but probably purely
technical, is used in the proof in a crucial way to construct the
forms $\alpha_\pm$.

Another application of Theorem \ref{thm:contact-sdm} considered in
\cite{GHHM} (and also in \cite{GGo, LL}) is the following result
originally proved in \cite{CGH}.

\begin{Theorem}
\label{thm:S3}
The Reeb flow of a contact form $\alpha$ supporting the standard
contact structure on $S^3$ has at least two simple closed orbits.
\end{Theorem}

In fact, a much stronger result holds. Namely, every Reeb flow on a
closed three-manifold has at least two simple closed Reeb orbits. This
fact is proved in \cite{CGH} using the machinery of embedded contact
homology and is outside the scope of this survey. The idea of the
proof from \cite{GHHM} is that if a Reeb flow on the standard contact
$S^3$ had only one simple closed orbit $x$, this orbit would be an
SDM, and, by Theorem \ref{thm:contact-sdm}, the flow would have
infinitely many periodic orbits. Showing that $x$ is indeed an SDM
requires a rather straightforward index analysis with one non-trivial
ingredient used to rule out a certain index pattern. In \cite{GHHM},
this ingredient is strictly three-dimensional and comes from the
theory of finite energy foliations (see \cite{HWZ1, HWZ2}). The
argument in \cite{GGo, LL} uses a variant of the resonance relation
for degenerate Reeb flows proved in \cite{GGo, LLW}. Theorem
\ref{thm:contact-sdm} can also be applied to give a simple proof,
based on the same idea, of the result from \cite{BL} that any Finsler
geodesic flow on $S^2$ has at least two closed geodesics; see
\cite{GGo}. (Of course, this fact also immediately follows from
\cite{CGH}.)

Interestingly, no multiplicity results along the lines of Theorem
\ref{thm:S3} have been proved in higher dimensions without restrictive
additional assumptions on the contact form. Conjecturally, every Reeb
flow on the standard contact sphere $S^{2n-1}$ has at least $n$ simple
closed Reeb orbits. This conjecture has been proved when the contact
form comes from a strictly convex hypersurface in $\R^{2n}$ and the
flow is non-degenerate or $2n\leq 8$; see \cite{LZ, Lo, Wa} and
references therein. In the degenerate strictly convex case, the lower
bound is $\lfloor n/2 \rfloor +1$. Without any form of a convexity
assumption, it is not even known if a Reeb flow on the standard
contact $S^5$ must have at least two simple closed orbits. It is easy
to see, however, that a non-degenerate Reeb flow on the standard
$S^{2n-1}$ has at least two simple closed orbits; see, e.g.,
\cite[Rmk.~3.3]{Gu:pr}.

We conclude this section by pointing out that the machinery of contact
homology which the proof of Theorem \ref{thm:CCC} relies on is yet to
be fully put on a rigorous basis.

\section{Twisted geodesic flows}
\label{sec:magnetic}
The results from Section \ref{sec:reeb} have immediate applications to
the dynamics of twisted geodesic flows. These flows give a Hamiltonian
description of the motion of a charge in a magnetic field on a
Riemannian manifold.

To be more precise, consider a closed Riemannian manifold $M$ and let
$\sigma$ be a closed 2-form (a \emph{magnetic field}) on $M$. Let us
equip $T^*M$ with the twisted symplectic structure
$\omega=\omega_0+\pi^*\sigma$, where $\omega_0$ is the standard
symplectic form on $T^*M$ and $\pi\colon T^*M\to M$ is the natural
projection, and let $K$ be the standard kinetic energy Hamiltonian on
$T^*M$ arising from the Riemannian metric on $M$.  The Hamiltonian
flow of $K$ on $T^*M$ governs the motion of a charge on $M$ in the
magnetic field $\sigma$ and is referred to as a \emph{twisted
  geodesic} or \emph{magnetic flow}. In contrast with the geodesic
flow (the case $\sigma=0$), the dynamics of a twisted geodesic flow on
an energy level depends on the level. In particular, when $M$ is a
surface of genus $g \geq 2$, the example of the horocycle flow shows
that a symplectic magnetic flow need not have periodic orbits on all
energy levels. Furthermore, the dynamics of a twisted geodesic flow
also crucially depends on whether one considers low or high energy
levels, and the methods used to study this dynamics further depend on
the specific properties of $\sigma$, i.e., on whether $\sigma$ is
assumed to be exact or symplectic.

The existence problem for periodic orbits of a charge in a magnetic
field was first addressed in the context of symplectic geometry by
V.I. Arnold in the early 80s; see \cite{Ar:fs, Ar88}. Namely, Arnold
proved that, as a consequence of the Conley--Zehnder theorem,
\cite{CZ:arnold}, a twisted geodesic flow on $\T^2$ with symplectic
magnetic field has periodic orbits on all energy levels when the
metric is flat and on all low energy levels for an arbitrary metric,
\cite{Ar88}. It is still unknown if the latter result can be extended
to all energy levels; however it was generalized to all surfaces in
\cite{Gi:FA}.

\begin{Example}
  Assume that $M$ is a surface and let $\sigma=f\,dA$, where $dA$ is
  an area form. Assume furthermore that $f$ has a non-degenerate
  critical point at $x$. Then it is not hard to see that essentially
  by the inverse function theorem the twisted geodesic flow on a low
  energy level has a closed orbit near the fiber over $x$.
\end{Example}

Since Arnold's work, the problem has been studied in a variety of
settings. We refer the reader to, e.g., \cite{Gi:newton} for more
details and references prior to 1996 and to, e.g., \cite{AMMP, CMP,
  GGM, Ke99} for some more recent results and references.

Here we focus exclusively on the case where the magnetic field form
$\sigma$ is symplectic (i.e., non-vanishing when $\dim M=2$), and we
are interested in the existence problem for periodic orbits on low
energy levels.  In this setting, in all dimensions, the existence of
at least one closed orbit with contractible projection to $M$
was proved in \cite{GG:wm, Us:tgf}.

Furthermore, when $\sigma$ is symplectic, we can also think of $M$ as
a symplectic submanifold of $(T^*M,\omega)$ and $K$ as a Hamiltonian
on $T^*M$ attaining a Morse--Bott non-degenerate local minimum $K=0$
at $M$. Thus we can treat the problem of the existence of periodic
orbits on a low energy level $P_\eps=\{ K=\eps\}$ as a generalization
of the classical Moser--Weinstein theorem (see \cite{Mo76, We}), where
an isolated non-degenerate minimum is replaced by a Morse--Bott
non-degenerate minimum and the critical set is symplectic. This is the
point of view taken in \cite{Ke99} and then in, e.g.,
\cite{GG:capacity, GG:wm}. To prove the existence of a periodic orbit
on every low energy level one first shows that almost all low energy
levels carry a periodic orbit with mean index in a certain range
depending only on $\dim M$ and having contractible projection to $M$;
see, e.g., \cite{GG:capacity, Sc}. This fact does not really require
$M$ to be symplectic; it is sufficient to assume that $\sigma\neq
0$. Then a Sturm theory type argument is used in \cite{GG:wm, Us:tgf}
to show that long orbits must necessarily have high index, and hence,
by the Arzela--Ascoli theorem, every low energy level carries a
periodic orbit. At this step, the assumption that the Hessian $d^2K$
is positive definite on the normal bundle to $M$ becomes essential,
cf.\ \cite[Sect.\ 2.4]{GG:capacity}.

There are also multiplicity results. Already in \cite{Ar:fs, Ar88}, it
was proved that when $M=\T^2$ and $\sigma$ is symplectic, there are at
least three (or four in the non-degenerate case) periodic orbits on
every low energy level $P_\eps$. Furthermore, Arnold also conjectured
that the lower bounds on the number of periodic orbits are governed by
Morse theory and Lusternik--Schnirelmann theory as in the Arnold
conjecture whenever $\sigma$ is symplectic and $\eps>0$ is small
enough. These lower bounds were then proved for surfaces in
\cite{Gi:FA}.

For the torus the proof is particularly simple. Let us fix a flat
connection on $P_\eps=\T^2\times S^1$.  When $\eps>0$ is small, the
horizontal sections are transverse to $X_K$, and one can show that the
resulting Poincar\'e return map is a Hamiltonian diffeomorphism
$\T^2\to\T^2$.  Now it remains to apply the Conley--Zehnder
theorem. Note that this argument captures only the short orbits, i.e.,
the orbits in the homotopy class of the fiber.  Likewise, the proof in
\cite{Gi:FA} captures only the orbits that stay close to a fiber and
wind around it exactly once. In higher dimensions, however, it is not
entirely clear how to define such short orbits. The difficulty arises
from the fact that $d^2K$ has several ``modes'' in every fiber, and
the modes can vary significantly and bifurcate from one fiber to
another. Furthermore, the Weinstein--Moser theorem provides a
hypothetical lower bound which is different from the one coming from
the Arnold conjecture perspective; see \cite{Ke99}. Without a
distinguished class of short orbits to work with, one is forced to
consider all periodic orbits and, already for $M=\T^2$, use the Conley
conjecture type results in place of the Arnold
conjecture. Hypothetically, as is observed in \cite{GG:wm}, every low
energy level should carry infinitely many simple periodic orbits, at
least when $(M,\sigma)$ is a symplectic CY manifold. This is still a
conjecture when $\dim M> 2$, but in dimension two the question has
been recently settled in \cite{GGM}. Namely, we have

\begin{Theorem}[\cite{GGM}]
\label{thm:magnetic}
Assume that $M$ is a surface of genus $g\geq 1$ and $\sigma$ is
symplectic. Then for every small $\eps>0$, the flow of $K$ has
infinitely many simple periodic orbits on $P_\eps$ with contractible
projections to $M$. Moreover, assume that the flow has finitely many
periodic orbits in the free homotopy class $\ff$ of the fiber. Then
for every sufficiently large prime $k$ there is a simple periodic
orbit in the class $\ff^k$, and all such classes are distinct.
\end{Theorem}

When $M=\T^2$, the theorem immediately follows from Arnold's cross
section argument once one uses the Conley conjecture for $\T^2$
(proved in \cite{FH}) instead of the Conley--Zehnder theorem; see
\cite{GG:wm}.  When $g\geq 2$, Theorem \ref{thm:magnetic} (almost)
follows from Theorem \ref{thm:CCC} since $P_\eps$ has contact type and
the flow is index--admissible as observed in \cite{Be}. The part that
is not a consequence of Theorem \ref{thm:CCC} is the existence of a
simple periodic orbit in the class $\ff^k$ for a large prime $k$
without any non-degeneracy assumptions. This is proved by applying the
second case of Theorem \ref{thm:contact-sdm} to the disjoint union
$P_\eps\sqcup P_E$, where $E$ is large, with the filling $W$ formed by
the part of $T^*M$ between these two energy levels, and $\fc=\ff$. The
proof of Theorem \ref{thm:magnetic} heavily relies on the machinery of
cylindrical contact homology via its dependence on Theorem
\ref{thm:CCC}. Note, however, that in the present setting one might be
able to circumvent foundational difficulties by using automatic
transversality results from~\cite{HN}. Alternatively, one could work
with the linearized contact homology or the equivariant symplectic
homology for the filling $W$.

Two difficulties arise in extending Theorem \ref{thm:magnetic} to
higher dimensions. One is that the energy levels do not have contact
type, and hence the standard contact or symplectic homology techniques
are not applicable. This difficulty appears to be more technical than
conceptual: using Sturm theory as in \cite{GG:wm} one can still
associate to a level a variant of symplectic homology generated by
periodic orbits on the level. A more serious obstacle is the lack of
filtration by the free homotopy classes $\ff^k$, which plays a central
role in the proof.

There seems to be no reason to expect Theorem \ref{thm:magnetic} to
hold for $S^2$. However, no counterexamples are known. For instance,
let us consider the round metric on $S^2$ and a non-vanishing magnetic
field $\sigma$ symmetric with respect to rotations about the
$z$-axis. The twisted geodesic flow on every energy level is
completely integrable. It would be useful and illuminating to analyze
this flow and check if it has infinitely many periodic orbits on every
(low or high) energy level.

It is conceivable that for any magnetic field, every sufficiently low
energy level carries infinitely many periodic orbits. For exact
magnetic fields on closed surfaces this is proved for almost all low
energy levels in \cite{AMMP} by (low-dimensional) methods of
``classical calculus of variations'', and it would be extremely
interesting to understand this phenomenon of ``almost existence of
infinitely many periodic orbits'' from a symplectic topology
perspective and generalize it to higher dimensions. Furthermore, even
in dimension two, no examples of magnetic flows with finitely many
periodic orbits on arbitrarily low energy levels are known. For
instance, it is not known if the completely integrable twisted
geodesic flow on $S^2$ with an exact $S^1$-invariant magnetic field
$\sigma$ has infinitely many periodic orbits on only almost all low
energy levels or in fact on all such levels. (Note that the
Katok--Ziller flows from \cite{Ka, Zi} correspond to high energy
levels.)

\section{Beyond the Conley conjecture}
\label{sec:beyond}

\subsection{Franks' theorem}
\label{sec:Franks}
Even when the Conley conjecture fails, the existence of infinitely
many simple periodic orbits is, as we have already seen, a generic
feature of Hamiltonian diffeomorphisms (and Reeb flows) for a broad
class of manifolds. There is, however, a different and more
interesting, from our perspective, phenomenon responsible for the
existence of infinitely many periodic orbits. The starting point here
is a celebrated theorem of Franks.

\begin{Theorem}[\cite{Fr1, Fr2}]
\label{thm:Fr}
Any area preserving diffeomorphism $\varphi$ (or, equivalently, a
Hamiltonian diffeomorphism) of $S^2$ with at least three fixed points
has infinitely many simple periodic orbits.
\end{Theorem}

In fact, the theorem, already in its original form, was proved for
homeomorphisms. This aspect of the problem is outside the scope of the
paper, and here we focus entirely on smooth maps. Furthermore, in the
setting of the theorem, there are also strong growth results; see
\cite{FH, LeC, Ke:JMD} for this and other refinements of Theorem
\ref{thm:Fr}. The proof of the theorem given in \cite{Fr1, Fr2}
utilized the methods from low--dimensional dynamics. Recently, a
purely symplectic topological proof of the theorem was obtained in
\cite{CKRTZ}; see also \cite{BH} for a different approach.

\begin{proof}[Outline of the proof from \cite{CKRTZ}] 
  Arguing by contradiction and passing if necessary to an iteration of
  $\varphi$, we can assume that $\varphi$ has finitely many periodic
  points, that these points are the fixed points and that there are at
  least three fixed points. Applying a variant of the resonance
  relations from \cite{GK09} combined with Theorem
  \ref{thm:sdm-conley} and a simple topological argument, it is not
  hard to see that there must be (at least) two fixed points $x$ and
  $y$ with irrational mean indices and at least one point $z$ with
  zero mean index.  Note that, since $\dim S^2=2$, the points $x$ and
  $y$ are elliptic and strongly non-degenerate, and $z$ is either
  degenerate or hyperbolic.

  In the former case, we glue together two copies of $S^2$ punctured
  at $y$ and $z$ by inserting narrow cylinders at the seams as in
  \cite[App.\ 9]{Ar}. As a result, we obtain a torus $\T^2$, and the
  Hamiltonian diffeomorphism $\varphi$ gives rise to an area
  preserving map $\psi\colon \T^2\to\T^2$. This map is not necessarily
  a Hamiltonian diffeomorphism, but it is symplectically isotopic to
  $\id$ and its Floer homology $\HF_*(\psi)$ is defined. Hence, either
  $\HF_*(\psi)=0$ or $\HF_*(\psi)\cong \H_{*+1}(\T^2)$ when $\psi$ is
  Hamiltonian.  Now one shows that, roughly speaking, any of the
  points $x^\pm\in \T^2$ corresponding to $x$ represents a non-trivial
  homology class of degree different from $0$ and $\pm 1$, which is
  impossible.

  When $z$ is a hyperbolic point, we use the points $x$ and $y$ to
  produce the torus, and again a simple Floer homological argument
  leads to a contradiction. Indeed, for a sufficiently large iteration
  of $\psi$, each elliptic point has large Conley--Zehnder index,
  since Theorem \ref{thm:sdm-conley} rules out SDM points, and each
  hyperbolic point has even index. Moreover, hyperbolic points give
  rise to non-trivial homology classes (cf.\ \cite[Thm.\
  1.7]{GG:generic}). Thus $\HF_*(\psi)\neq 0$ but $\HF_1(\psi)=0$,
  which is again impossible.  (Alternatively, one can just apply
  Theorem \ref{thm:hyperbolic} below to deal with this
  case.) \end{proof}

\subsection{Generalizing Franks' theorem}
\label{sec:franks-gen} 
Even though all proofs of Franks' theorem are purely low-dimensional,
it is tempting to think of the result as a particular case of a more
general phenomenon. For instance, one hypothetical generalization of
Franks' theorem would be that a Hamiltonian diffeomorphism with “more
than necessary” non-degenerate (or just homologically non-trivial in
the sense of Section \ref{sec:FH}) fixed points has infinitely many
periodic orbits. Here “more than necessary” is usually interpreted as
a lower bound arising from some version of the Arnold conjecture. For
$\CP^n$, the expected threshold is $n + 1$ and, in particular, it is
$2$ for $S^2$ as in Franks' theorem, cf.\ \cite[p.\ 263]{HZ}.

However, this conjectural generalization of Franks' theorem seems to
be too restrictive, and from the authors' perspective it is fruitful
to put the conjecture in a broader context. Namely, it appears that
the presence of a fixed point that is unnecessary from a homological
or geometrical perspective is already sufficient to force the
existence of infinitely many simple periodic orbits. Let us now state
some recent results in this direction.

\begin{Theorem}[\cite{GG:hyperbolic}]
\label{thm:hyperbolic}
A Hamiltonian diffeomorphism of $\CP^n$ with a hyperbolic periodic
orbit has infinitely many simple periodic orbits.
\end{Theorem}

Here, clearly, the hyperbolic periodic orbit is unnecessary from every
perspective. In contrast with Franks' theorem and the Conley
conjecture type results, as of this writing, there are no growth
results in this setting. The theorem actually holds for a broader
class of manifolds $M$, and the requirements on $M$ can be stated
purely in terms of the quantum homology of $M$; see \cite[Thm.\
1.1]{GG:hyperbolic}. Among the manifolds meeting these requirements
are, in addition to $\CP^n$, the complex Grassmannians $\Gr(2; N)$,
$\Gr(3; 6)$ and $\Gr(3; 7)$; the products $\CP^m\times 􏰺 P^{2d}$ and
$\Gr(2; N)\times P^{2d}$, where $P$ is symplectically aspherical and
$d\leq m$ in the former case and $d\leq 2$ in the latter; and the
monotone products $\CP^m\times \CP^m$. There is also a variant of the
theorem for non-contractible hyperbolic orbits, which is applicable
to, for example, the product $\CP^m\times P^{2d}$. Note also that the
generalization of Franks' theorem to $\CP^n$, at least for
non-degenerate Hamiltonian diffeomorphisms, would follow if one could
replace in Theorem \ref{thm:hyperbolic} a hyperbolic fixed point by a
non-elliptic one.

Another result fitting into this context is the following.

\begin{Theorem}[\cite{Gu:hq}]
\label{thm:hq}
Let $\varphi\colon \R^{2n}\to\R^{2n}$ be a Hamiltonian diffeomorphism
generated by a Hamiltonian equal to a hyperbolic quadratic form $Q$ at
infinity (i.e., outside a compact set) such that $Q$ has only real
eigenvalues. Assume that $\varphi$ has finitely many fixed points, and
one of these points, $x$, is non-degenerate (or just isolated and
homologically non-trivial) and has non-zero mean index. Then $\varphi$
has simple periodic orbits of arbitrarily large period.
\end{Theorem}

As a consequence, regardless of whether the fixed-point set is finite
or not, $\varphi$ has infinitely many simple periodic orbits. In this
theorem the condition that the eigenvalues of $Q$ are real can be
slightly relaxed. Conjecturally, it should be enough to require $Q$ to
be non-degenerate and $x$ to have mean index different from the mean
index of the origin for $Q$. However, hyperbolicity of $Q$ is used in
an essential way in the proof of the theorem. Also, interestingly, in
contrast with Franks' theorem, the requirement that $x$ is
homologically non-trivial is essential and cannot be omitted, even in
dimension two. As an easy consequence of Theorem \ref{thm:hq}, we
obtain

\begin{Theorem}
[\cite{Gu:hq}]
\label{thm:hq2}
Let $\varphi\colon \R^{2n}\to\R^{2n}$, where $2n=2$ or $2n=4$, be a
Hamiltonian diffeomorphism generated by a Hamiltonian equal to a
hyperbolic quadratic form $Q$ at infinity as in Theorem
\ref{thm:hq}. Assume that $\varphi$ is strongly non-degenerate and has
at least two (but finitely many) fixed points. Then $\varphi$ has
simple periodic orbits of arbitrarily large period.
\end{Theorem}

In the two-dimensional case, a stronger result is proved in
\cite[Thm.\ 5.1.9]{Ab}.  In the setting of Theorems \ref{thm:hq} and
\ref{thm:hq2}, one can be more precise about which simple periods
occur. Namely, for a certain integer $m>0$, starting with a
sufficiently large prime number, among every $m$ consecutive primes,
there exists at least one prime which is the period of a simple
periodic orbit. Thus, as in many other results of this type, we have
the growth lower bound $\const\cdot k/\ln k$.

Theorem \ref{thm:hyperbolic} and, with some extra work, Theorem
\ref{thm:hq} imply the case of Franks' theorem where $\varphi$ is
assumed to have a hyperbolic periodic orbit, e.g., when $\varphi$ is
non-degenerate. Furthermore, it is conceivable that one could prove
Franks' theorem as a consequence of Theorem \ref{thm:hyperbolic}. Such
a proof would certainly be of interest, but it would most likely be
much more involved than the argument in \cite{CKRTZ}.

Let us now turn to non-contractible orbits. Recall that a
(time-dependent) Hamiltonian flow $\varphi_H^t$ generated by a
Hamiltonian $H\colon S^1\times M \to \R$, there is a one-to-one
correspondence between one-periodic orbits of $\varphi_H^t$ and the
fixed points of $\varphi=\varphi_H$. Furthermore, as is easy to see
from the proof of the Arnold conjecture, the free homotopy class of an
orbit $x$ is independent of the Hamiltonian generating the time-one
map $\varphi$. Thus the notion of a contractible one-periodic orbit
(or even a ``contractible fixed point'') of $\varphi$ is
well-defined. Of course, the same applies to $k$-periodic orbits.

On a closed symplectic manifold $M$ a Hamiltonian diffeomorphism need
not have non-contractible one-periodic orbits. Indeed, the Hamiltonian
Floer homology vanishes for any non-trivial free homotopy class when
$M$ is compact, since all one-periodic orbits of a $C^2$-small
autonomous Hamiltonian are its critical points (hence contractible).
Thus, from a homological perspective, non-contractible periodic orbits
are totally unnecessary.

To state our next result, recall that a symplectic form $\omega$ on
$M$ is said to be atoroidal if for every map $v\colon \T^2 \to M$, the
integral of $\omega$ over $v$ vanishes. We have

\begin{Theorem}[\cite{Gu:nc}]
\label{thm:nc}
Let $M$ be a closed symplectic manifold equipped with an atoroidal
symplectic form $\omega$.  Assume that a Hamiltonian diffeomorphism
$\varphi$ of $M$ has a non-degenerate one-periodic orbit $x$ with
homology class $[x]\neq 0$ in $\H_1(M;\Z)/\Tor$ and that the set of
one-periodic orbits in the class $[x]$ is finite. Then, for every
sufficiently large prime $p$, the Hamiltonian diffeomorphism $\varphi$
has a simple periodic orbit in the homology class $p[x]$ and with
period either $p$ or $p'$, where $p'$ is the first prime greater than
$p$.
\end{Theorem}

Thus the number of simple non-contractible periodic orbits with period
less than or equal to $k$, or the number of distinct homology classes
represented by such orbits, is bounded from below by $\const \cdot k/
\ln k$.  An immediate consequence of the theorem is that $\varphi$ has
infinitely many simple periodic orbits with homology classes in
$\N[x]$ whether or not the set of one-periodic orbits (in the class
$[x]$) is finite. Moreover, in this theorem, as in Theorem
\ref{thm:hq}, the non-degeneracy condition, can be relaxed and
replaced by a much weaker requirement that $x$ is isolated and
homologically non-trivial.  Finally, in both of the theorems, the
orbit $x$ need not be one-periodic; the theorems (with obvious
modifications) still hold when $x$ is just a periodic orbit.

Among the manifolds meeting the requirements of Theorem \ref{thm:nc}
are, for instance, closed K\"ahler manifolds with negative sectional
curvature and, more generally, any closed symplectic manifold with
$[\omega]|_{\pi_2(M)}=0$ and hyperbolic $\pi_1(M)$. Furthermore,
Hamiltonian diffeomorphisms having a periodic orbit in a non-trivial
homology class exist in abundance. It is plausible that a
$C^\infty$-generic Hamiltonian diffeomorphism has an orbit in a
non-trivial homology class when the fundamental group (or the first
homology group) of $M$ is large enough. However, already for $M=\T^2$,
a fixed Hamiltonian diffeomorphism need not have non-contractible
periodic orbits, and even $C^\infty$-generically one cannot prescribe
the homology class of an orbit in advance.  

Hypothetically, one can expect an analog of the theorem to hold when
the condition that $\omega$ is atoroidal is omitted or relaxed, e.g.,
replaced by the requirement that $(M,\omega)$ is toroidally monotone.
  
The proofs of all these theorems are based on the same idea that an
unnecessary periodic orbit is a seed creating infinitely many periodic
orbits. In Theorems \ref{thm:hq} and~\ref{thm:nc} the argument is
that, roughly speaking, the change in filtered Floer homology, for a
carefully chosen action range (and/or degree), between different
iterations of $\varphi$ requires new simple periodic orbits to be
created. The proof of Theorem \ref{thm:hyperbolic} relies on a result,
perhaps of independent interest, asserting that the energy needed for
a Floer connecting trajectory of an iterated Hamiltonian to approach a
hyperbolic orbit and cross its fixed neighborhood cannot become
arbitrarily small: it is bounded away from zero by a constant
independent of the order of iteration. Then the product structure in
quantum homology is used to show that there must be Floer connecting
trajectories with energy converging to zero for some sequence of
iterations unless new periodic orbits are created.

\subsection{Reeb flows, symplectomorphisms and all that}
\label{sec:symplecto}
The conjectures discussed in Section \ref{sec:franks-gen} have obvious
analogs for Reeb flows and symplectomorphisms.

\subsubsection{Reeb flows revisited}
Just like Hamiltonian diffeomorphisms, Reeb flows with ``unnecessary''
periodic orbits can be expected to have infinitely many simple
periodic orbits. However, as of this writing, there is little evidence
supporting this conjecture, and all the relevant results are
three-dimensional.  The most notable one is a theorem, proved in
\cite{HWZ:convex}, asserting that the Reeb flow on a strictly convex
hypersurface in $\R^4$ has either two or infinitely many periodic
orbits. In fact, more generally, this is true for the so-called
dynamically convex contact forms on $S^3$.  Conjecturally, this
``two-or-infinitely-many'' alternative should hold for all contact
forms supporting the standard contact structure on $S^3$, which could
be thought of as a three-dimensional analog of Franks' theorem; see
\cite{HWZ03} for some other related results.

The existence of infinitely many closed geodesics on $S^2$ also fits
perfectly into the framework of this conjecture; see \cite{Ba, Fr1}
and also \cite{Hi:geod, Hi:length} and the references therein for the
original argument. Indeed, the classical Lusternik--Schnirelmann
theorem asserts the existence of at least three closed geodesics on
$S^2$, i.e., at least one more than is necessary from the
Floer--theoretic perspective, cf.\ \cite{Ka, Zi}. The geodesic flow on
$S^2$, interpreted as a Reeb flow on the standard contact $\RP^3$,
should then have infinitely many simple closed orbits, i.e., simple
closed geodesics on $S^2$. In fact, one can prove the existence of
infinitely many closed geodesics on $S^2$ in exactly this way using
the variant of the Lusternik--Schnirelmann theorem from \cite{Gr} as
the starting point and then the results from \cite{HMS} and
\cite{GHHM} on the symplectic side of the problem; see the latter
reference for more details.

Finally, another aspect of this question is related to the so-called
perfect Reeb flows. Let us call a non-degenerate Reeb flow on a
contact manifold \emph{perfect} if the differential in the contact
homology complex vanishes for some choice of the auxiliary data, cf.\
\cite{BCE}. (Thus this definition depends on the type of the contact
homology used.) For instance, a Reeb flow is perfect (for every
auxiliary data) when all closed orbits have Conley--Zehnder index of
the same parity; we refer the reader to \cite{Gu:pr} for numerous
examples of perfect Reeb flows. One can think of non-perfect Reeb
flows as those with unnecessary periodic orbits. In \cite{Gu:pr} an
upper bound on the number of simple periodic orbits of perfect Reeb
flows is established for many contact manifolds under some (minor)
additional assumptions.  For $S^{2n-1}$, as expected, the upper bound
is $n$.  However, in general, it is not even known if a perfect Reeb
flow on the standard contact $S^{2n-1}$, $2n-1 \geq 5$, must have
finitely many simple periodic orbits or, if it does, whether this
number is independent of the flow. (For $S^3$, this is proved in
\cite{BCE} and reproved in \cite{Gu:pr}.)

\subsubsection{Symplectomorphisms} For symplectomorphisms, the problem
of the existence of infinitely many periodic orbits breaks down into
several phenomena in the same way as for Reeb flows, although even
less is known. Namely, as in Section \ref{sec:reeb-general}, we can,
roughly speaking, single out three types of behavior of
symplectomorphisms.  First of all, some manifolds (such as $\CP^n$ or
tori or their products) admit symplectomorphisms with finitely many
periodic orbits or even, in some instances (e.g., $\T^{2n}$), without
periodic orbits.

Then there are symplectomorphisms $\varphi$ such that the rank of the
Floer homology $\HF_*(\varphi^k)$ over a suitable Novikov ring
$\Lambda$ grows with the order of iteration $k$.  The Floer homology
groups of symplectomorphisms have been studied for close to two
decades starting with \cite{DS, LO}, and the literature on the subject
is quite extensive (particularly so for symplectomorphisms of
surfaces); we refer the reader to, e.g., \cite{CC, Fe} and references
therein for recent results focusing specifically on the growth of the
Floer homology.  Let us assume here, for the sake of simplicity, that
$M$ is symplectically aspherical or monotone and that the Floer
homology $\HF_*(\varphi^k)$ is defined.  Similarly to the results in
\cite{GM, HM, McL}, we have

\begin{Proposition}
\label{prop:sympl}
Let $\varphi\colon M\to M$ be a symplectomorphism of a closed
symplectic manifold $M$ such that the sequence $\rk_\Lambda
\HF_*(\varphi^k)$ is unbounded. Then $\varphi$ has infinitely many
simple periodic orbits. Moreover, every sufficiently large prime 
occurs as a simple period when the number of fixed points of $\varphi$
is finite and $\rk_\Lambda\HF_*(\varphi^k)\to \infty$.
\end{Proposition}

\begin{proof}
  The proposition is obvious and well known when $\varphi$ is strongly
  non-degenerate. (See \cite{CC, Fe} for more specific and stronger
  results.) The degenerate case follows from the fact that the
  dimension of the local Floer homology of an isolated periodic orbit
  remains bounded as a function of the order of iteration, as a
  consequence of Theorem~\ref{thm:persist-lf}.
\end{proof}

\begin{Example} Let $\Sigma$ be a closed surface and $\psi\colon
  \Sigma\to\Sigma$ be a symplectomorphism such that $\rk_\Lambda
  \HF_*(\psi^k)\to\infty$. This is the case, for instance, when the
  Lefschetz number $L(\psi^k)$ grows; such symplectomorphisms exist in
  abundance. (Proposition \ref{prop:sympl} applies to $\psi$, but in
  this case a simpler argument is available: when $L(\psi^k)\to\infty$
  the assertion immediately follows from the Shub--Sullivan theorem,
  \cite{SS}.) Let $P$ be a symplectically aspherical manifold with
  $\chi(P)=0$, such as $P=\T^{2n}$, and $\varphi\colon P\times
  \Sigma\to P\times \Sigma$ be Hamiltonian isotopic to
  $(\id,\psi)$. Then $\rk_\Lambda \HF_*(\varphi^k)\to\infty$ and, by
  the proposition, $\varphi$ has infinitely many simple periodic
  orbits. However, $L(\varphi^k)=0$, and, moreover, a
  symplectomorphism $\varphi$ in the smooth or symplectic isotopy
  class of $(\id,\psi)$ need not have periodic orbits at all when,
  e.g., $P=\T^{2n}$.
\end{Example}

Thirdly, there are symplectomorphisms with infinitely many simple
periodic orbits, but no homological growth. Here, of course, we have
the Hamiltonian Conley conjecture as a source of examples.  The
authors tend to think that there should be numerous other classes of
symplectomorphisms of this type, but no results to this account have
so far been proved. One class of symplectomorphisms for which the
symplectic Conley conjecture is likely to hold and probably within
reach is that of symplectomorphisms of a surface of genus $g\geq 2$
symplectically isotopic to~$\id$.

Finally, one can expect the presence of an unnecessary fixed or
periodic point to force a symplectomorphism to have infinitely many
simple periodic orbits. However, now the situation is more subtle,
less is known, and there is a counterexample to this general
principle. A prototypical (and simple) result of this type is that a
non-degenerate symplectomorphism of $\T^2$ symplectically isotopic to
$\id$ has infinitely many simple periodic orbits, provided that it has
one fixed or periodic point; see \cite[Thm.\ 1.7]{GG:generic}. In
other words, we have the following ``zero or infinitely many''
alternative: a non-degenerate symplectomorphism of $\T^2$ isotopic to
$\id$ has either no periodic orbits or infinitely many periodic
orbits. It is interesting, however, that the non-degeneracy condition
cannot entirely be omitted, although it can probably be
relaxed. Namely, it is easy to construct a symplectic vector field on
$\T^2$ with exactly one (homologically trivial) zero and no periodic
points; see \cite[Example 1.10]{GG:generic}. (No similar results or
counterexamples for tori $\T^{2n}$, $2n\geq 4$, are known.)  There is
an analog of Theorem \ref{thm:hyperbolic} for symplectomorphisms,
\cite{Bat}, applicable to manifolds such as $\CP^n\times P^{2m}$,
where $P$ is symplectically aspherical and $m\leq n$.

Note in conclusion that when discussing symplectomorphisms in the
homological framework, it would make sense to ask the question of the
existence of infinitely many periodic orbits while fixing the class of
symplectomorphisms Hamiltonian isotopic to each other. The reason is
that the Floer homology is very sensitive to symplectic isotopy but is
invariant under Hamiltonian isotopy.  Above, however, we have not
strictly adhered to this point of view and mainly focused on the
properties of the ambient manifolds.

\begin{Remark}
  In this survey, we have just briefly touched upon the question of
  the existence of infinitely many periodic orbits for Hamiltonian
  diffeomorphisms and symplectomorphisms of open manifolds and
  manifolds with boundary. (In this case, one, of course, has to
  impose some restrictions on the behavior of the map near the
  infinity or on the boundary.) Such symplectomorphisms naturally
  arise in applications and in physics. For instance, the billiard
  maps and the time-one maps describing the motion of a particle in a
  time-dependent conservative force field and/or exact magnetic field
  are in this class. We are not aware of any new phenomena happening
  in this setting, and our general discussion readily translates to
  such maps. For instance, Hamiltonian diffeomorphisms of open
  manifolds can exhibit the same three types of behavior as
  symplectomorphisms of closed manifolds or Reeb flows. (After all, a
  geodesic flow is a Hamiltonian flow on the cotangent bundle.)  To
  the best of our knowledge, there are relatively few results of
  symplectic topological nature concerning this class of maps; see
  Section \ref{sec:CC-history} for some of the relevant references.
  \end{Remark}

\end{document}